%% file: main.tex
\documentclass[twoside]{amsart}
\usepackage{ulem} %\fk
\usepackage{latexsym}
\usepackage{amssymb,amsmath,amsopn}
\usepackage[dvips]{graphicx}   %To insert ps figures
\usepackage{color,epsfig}      %To insert ps figures

               {\begin{list}{}{\leftmargin#1\rightmargin#2}\item{}}%
               {\end{list}}

\newtheorem{thm}{Theorem}
\newtheorem{lem}{Lemma}
\newtheorem{prop}{Proposition}
\theoremstyle{definition}
\newtheorem{defn}{Definition}
\newtheorem{rem}{Remark}

\newtheorem{cor}{Corollary}

\definecolor{green}{rgb}{0.0, 0.75, 0.25} %\fk
\def\eea{\end{eqnarray}}
\renewcommand{\emph}{\textit}

\DeclareMathOperator{\bd}{bd}

\DeclareMathOperator{\sign}{sign}

\begin{document}

%%%%%%%%%%%%%%%%%%%%%%%

\title[Mono-monostatic polyhedra with point masses have at least 8 vertices]{Mono-monostatic polyhedra with uniform point masses have at least 8 vertices}

\author[S. Boz\'oki, G. Domokos, F. Kov\'acs and K. Reg\H{o}s] {S\'andor Boz\'oki, G\'abor Domokos, Fl\'ori\'an Kov\'acs and Krisztina Reg\H{o}s}
\address{G\'abor Domokos, MTA-BME Morphodynamics Research Group and Dept. of Mechanics, Materials and Structures, Budapest University of Technology,
M\H uegyetem rakpart 1-3., Budapest, Hungary, 1111}
\email{domokos@iit.bme.hu}
\address{Fl\'ori\'an Kov\'acs, MTA-BME Morphodynamics Research Group and Dept. of Structural Mechanics, Budapest University of Technology, M\H uegyetem rakpart 1-3., Budapest, Hungary, 1111}
\email{kovacs.florian@emk.bme.hu}
\address{Krisztina Reg\H os, MTA-BME Morphodynamics Research Group,
M\H uegyetem rakpart 1-3., Budapest, Hungary, 1111}
\email{regos.kriszti@gmail.com}
\address{S\'andor Boz\'oki, Research Group of Operations Research and Decision Systems, Research Laboratory on Engineering and Management Intelligence, Institute for Computer Science and Control (SZTAKI), E\"otv\"os Lor\'and Research Network (ELKH), Kende utca 13-17, Budapest, Hungary, 1111}
\email{bozoki.sandor@sztaki.hu}

\thanks{Support of the NKFIH Hungarian Research Fund grant 134199 and of the TKP2020 IES Grant No. TKP2020  BME-IKA-VIZ is kindly acknowledged. The research of S. Boz\'oki was supported by the Hungarian National Research, Development and Innovation Office (NKFIH) under Grant NKFIA ED\_18-2-2018-0006.}
\subjclass[2010]{52B10, 77C20, 52A38}

\keywords{polyhedron, static equilibrium, monostatic polyhedron, polynomial inequalities}

\maketitle

%\tableofcontents

\begin{abstract}
The monostatic property of convex polyhedra (i.e. the property of having just one stable or unstable static equilibrium point) has been in the focus of research ever since Conway and Guy \cite{Conway} published the proof of the existence of the first such object, followed by the constructions of Bezdek \cite{Bezdek} and Reshetov \cite{Reshetov}. These examples establish $F\leq 14, V\leq 18$ as the respective \emph{upper bounds} for the minimal number of faces and vertices for a homogeneous mono-stable polyhedron. By proving that no mono-stable homogeneous tetrahedron existed, Conway and Guy \cite{Conway} established for the same problem the lower bounds for the number of faces and vertices as $F, V \geq 5$  and the same lower bounds were also established \cite{balancing} for the mono-unstable case. It is also clear that the $F,V \geq 5$ bounds also apply for convex, homogeneous point sets with unit masses at each point (also called polyhedral 0-skeletons) and they are also valid for mono-monostatic polyhedra
with exactly on stable and one unstable equilibrium point (both homogeneous and 0-skeletons). Here we present an algorithm by which we improve the lower bound to $V\geq 8$ vertices (implying $f \geq 6$ faces)  on mono-unstable and mono-monostable 0-skeletons. Our algorithm appears to be less well suited to compute the lower bounds for mono-stability. We point out these difficulties in connection with the work of
Dawson and Finbow \cite{Dawson, Dawson2, Dawson3} who explored the monostatic property of simplices in higher dimensions. \end{abstract}

\section{Introduction}\label{sec:intro}

\subsection{Mono-stability and polyhedra}
The minimal number of positions where a rigid body could be at rest, also called static equilibria, have been investigated at least since Archimedes developed his famous design for ships \cite{Archimedes}. Having one stable position (i.e. being \emph{mono-stable}) is of obvious advantage for ships, however, for rigid bodies under gravity, supported on a rigid (frictionless) surface this property, facilitating self-righting,  could also be  of advantage. 

Beyond being useful, mono-stable bodies also appear to be of particular mathematical interest, because it is still unclear
what are the minimal numbers $F^S, V^S$ of faces and vertices for a homogeneous, convex,  mono-stable polyhedron. The first such object with $F=19$ faces and $V=34$ vertices has been described by Conway and Guy in 1967 \cite{Conway}. This construction was improved by Bezdek \cite{Bezdek} to $(F,V)=(18,18)$ and later by Reshetov \cite{Reshetov} to $(F,V)=(14,24)$. These values define the best known \emph{upper bounds}, so we have  $F^S \leq 14 , V^S \leq 18$. Even less is known about the lower bounds: the only result is due to Conway \cite{Dawson} who proved that a homogeneous tetrahedron has at least two stable equilibria, so we have $F^S, V^S \geq 5$. The gap between the upper and lower bounds (illustrated in Figure \ref{fig:1}) is substantial. 

\begin{figure}[ht]
\begin{center}
\includegraphics[width=\columnwidth]{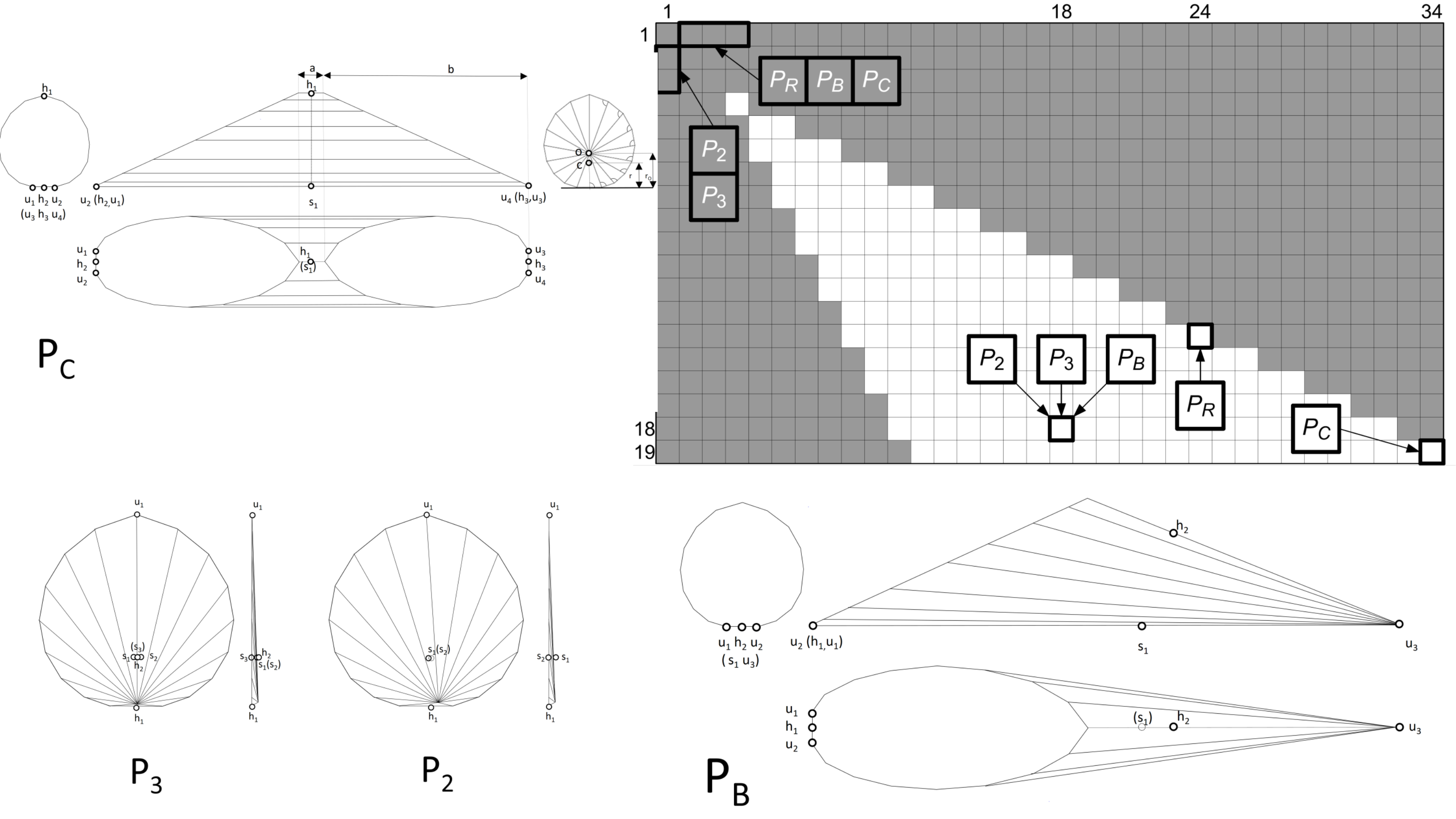}
\caption{Illustration of some monostatic polyhedra and their location on the overlay of the  $[F,V]$ (face and vertex number) grid  and $[S,U]$ (number of stable and unstable equilibria) grid: Conway's polyhedron \cite{Conway}
$P_C$ in classes $(F,V)=(19,34), (S,U)=(1,4)$,  Bezdek's  polyhedron \cite{Bezdek}, $P_B$ in classes $(F,V)=(18,18), (S,U)=(1,3)$, Reshetov's polyhedron, \cite{Reshetov} $P_R$ in classes $(F,V)=(14,24), (S,U)=(1,2)$. Mono-unstable polyhedra $P_2,P_3$ in classes  $(F,V)=(18,18)$, $(S,U)=(1,2), (1,3)$. White squares correspond on the $[F,V]$ grid to combinatorial classes where, according to Steinitz \cite{Steinitz1, Steinitz2} we find polyhedra, on the $[S,U]$ grid the same squares correspond to equilibrium classes where we find polyhedra with $S=F, U=V$ \cite{balancing}. }
\label{fig:1}
\end{center}
\end{figure}

\subsection{Generalizations}\label{ss:gen}
The above problem has many generalizations of which we list a few below and we also briefly summarize the related results:

\begin{enumerate}

\item Instead of looking at homogeneous polyhedra, one may look at \emph{$h$-skeletons} which are polyhedra with mass uniformly distributed on their $h$-dimensional faces (see Figure \ref{fig:material}). If $h=3$ then we have the original problem (homogeneous polyhedron). A 0-skeleton is a homogeneous point set, with unit mass at each point (i.e. a polyhedron with unit masses at the vertices). We will denote the minimal facet and vertex numbers for mono-stable $h$-skeletons by $F^S_h, V^S_h$, respectively. In case of simplices, the center of mass of a homogeneous body and a 0-skeleton coincide \cite{KrantzMcCarthyParks} , so Conway's result for the tetrahedron implies that for 0-skeletons we also have  $F^S_0, V^S_0 \geq 5$ and upper bounds are not known. For simplicity and consistency of notation, in the homogeneous case we will drop the lower index and use $F^S,S^S$, as before.

\begin{figure}[ht]
\begin{center}
\includegraphics[width=0.8\columnwidth]{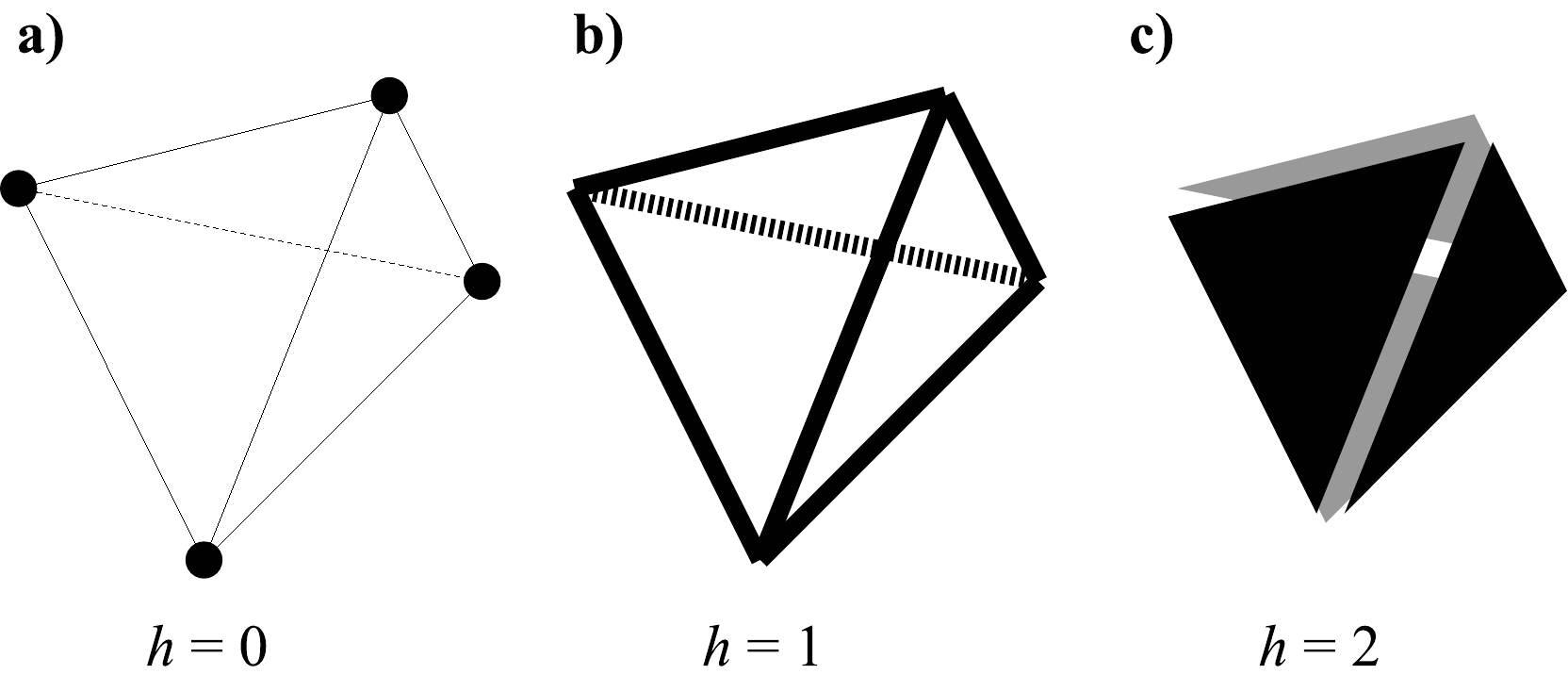}
\caption{Illustration of homogeneous material distributions for polyhedra: (a) 0-skeleton, (b) 1-skeleton, (c) 2-skeleton. The case $h=3$, i.e., uniform mass distribution over the entire volume is not shown for simple graphical reasons.}
\label{fig:material}
\end{center}
\end{figure}

\item Instead of looking at mono-stable polyhedra one may look at \emph{mono-unstable} polyhedra, i.e. for polyhedra which have only one vertex on which they can be balanced in an (unstable) equilibrium position. For the homogeneous case, such polyhedra were first exhibited in \cite{balancing} with an example at $(F,V)=(18,18).$  In the same paper, Conway's result for the nonexistence of a mono-stable tetrahedron was extended to the nonexistence of a mono-unstable polyhedron, so for the minimal numbers $F^U, V^U$ for the faces and vertices of a mono-unstable polyhedron we have $5 \leq F^U, V^U \leq 18$ and for 0-skeletons we have  $5 \leq F^U, V^U$.

\item Instead of looking at  polyhedra one may look at polytopes in higher dimensions. Dawson and Finbow \cite{Dawson, Dawson2, Dawson3} published a series of papers where they proved that in $d\geq 10$ dimensions mono-stable simplices exist, thus they proved for $d\geq 10$ that $F^S= F^S_0=d+1$ and they also proved that for $d<9$ we have $F^S,F^S_0 \geq d+2$

\item Instead of looking at mono-stable or mono-unstable objects (collectively called  monostatic) one may look at \emph{mono-monostatic} objects which have just one stable and one unstable equilibrium point. Among piecewise smooth 
convex bodies the first such object, called G\"omb\"oc, was identified in 2007 \cite{VarkonyiDomokos}. Little is known about the
minimal number of faces and and vertices for a polyhedron in this class. These lower bounds  we will denote by $F^*_h, V^*_h$. Obviously, we have

\begin{equation}\label{e:1}
F^*_h \geq F^S_h, F^U_h, \quad V^*_h \geq V^S_h, V^U_h,
\end{equation}
however, despite a related prize \cite{balancing}, there are no other bounds known for either $F^*_h$ or $V^*_h$.

\end{enumerate}

\subsection{Main result, strategy of the proof and outline of the paper}
Our goal is to prove the statement claimed in the title of the paper.  In fact, we will prove a stronger claim which, using the current notation can be stated as
\begin{thm}\label{thm:7}
$V^U_0 \geq 8$.
\end{thm}
Using (\ref{e:1}), it is immediately clear that from Theorem \ref{thm:7} we have
\begin{lem}\label{lem:main}
    $V^*_0 \geq 8$
\end{lem}
which is the statement claimed in the title of the paper. We will  prove Theorem \ref{thm:7} in four steps:

\begin{enumerate}
    \item In section \ref{s:main} we prove Theorem \ref{thm:main} which connects the vertex geometry (i.e. the collection of the $V$ vectors pointing from the center of mass to the vertices) of a polytope,
     to the number $U$  of its unstable equilibria. (We will also introduce the dual statement for faces and stable equilibria but we will not use the latter in the current paper for computations.)
     
     \item In section \ref{s:alg} we show that if $U=1$ then Theorem \ref{thm:main} can be converted into a system of 
     quadratic inequalities and we construct the general system for $d=3$-dimensional 0-skeletons. The Theorem
     also admits to develop such systems for other cases, such as homogeneous polyhedra, or polytopes in other dimensions, however, in the current paper we only investigate 3 dimensional 0-skeletons.
     
     \item In section \ref{sec:ex} we develop and solve this system for the tetrahedron, serving as an illustrative example of our method. Here we prove that $V^0 \geq 5$ in a manner which is independent of the original proof of Conway.

     \item In section \ref{sec:No-mono} we develop an optimization algorithm which can prove the unsolvability of  the system of quadratic inequalities constructed in Section \ref{s:alg} by doing exact (rational) arithmetic and by using this code we prove Theorem \ref{thm:7}. 
\end{enumerate}

The same algorithm could, in principle, also \emph{solve} such a system (instead of proving its unsolvability). If the solution for $V=\bar V$ was convex and all systems with $V< \bar V$ were proven to be unsolvable then, in this case the algorithm would provide the \emph{exact} value as $V^U_0=\bar V$.  Since we did not find such a solution, Theorem \ref{thm:7} provides a lower bound for $V^U_0$. Analogous computations for other mass distributions (e.g. the homogeneous case) and other dimensions are, in principle, possible but have not been performed in this paper.
Section \ref{sec:con} discusses why the dual case (i.e. the algorithm to compute $F^S$) is not pursued in the current paper and we establish the link between our formulae and those of Dawson \cite{Dawson}. We also draw some conclusions.

\section{The geometric construction}\label{s:main}

\subsection{Problem statement}\label{ss:problemstatement}
Let $P$ be a $d$-dimensional convex polytope with $F$ faces $f_i$,
$(i=1,2, \dots F)$ and $V$ vertices $v_i$, $(i=1,2, \dots V)$.

We identify the faces by the \emph{face vectors} $\mathbf{q}_i$
orthogonal to face $f_i$ with $|\mathbf{q}_i|$ being equal to the orthogonal distance
of the face $f_i$ from the center of mass $o$.

The vertices are identified by \emph{vertex vectors} $\mathbf{r}_i$, ($i=1,2, \dots V$) with origin at the center of mass $o$ and we assume $P$ to be generic in the sense that if $i \not=j$ then $|\mathbf{r}_i| \not = |\mathbf{r}_j|$.

\begin{defn}\label{def:unstable}
We say that $p_i$ is an (unstable) static equilibrium point of $P$ if the support plane $Z$ of $P$ at $p_i$ with respect to $o$ does not intersect $P$. More precisely: we speak about a \emph{nondegenerate} equilibrium point $p_i$ if $Z \cup P = p_i$, and the equilibrium is degenerate if $Z$ also contains some edges of $P$. The term `unstable equilibrium' refers to a nondegenerate one henceforth unless stated otherwise.
We denote the number of unstable static equilibrium points of $P$ by $U(P)$ and we have $U(P) \leq V$.
\end{defn}

\noindent Our main result connects the vertex geometry (i.e. the set of position vectors $\mathbf{r}_i$) of $P$ with
the number $U(P)$ of its unstable static equilibria via the following formula:

\begin{thm}\label{thm:main} Let $P$ be a convex polytope. Then
\begin{equation}\label{main}
    U(P)=\sum_{i=1}^V  \left \lfloor \frac{1}{2}+ \frac{\sum_{j=1}^V \sign\left((\mathbf{r}_i-\mathbf{r}_j)\mathbf{r}_i\right)}{2(V-1)}\right \rfloor.
\end{equation}
\end{thm}

We also phrase an analogous, though more restricted result for stable equilibria.

\begin{defn}\label{def:stable}
We say that face $f_i$ carries a (nondegenerate) stable static equilibrium point of $P$ if the orthogonal projection $p_i$ of the center of mass $o$ onto the plane of $f_i$ is contained in the interior of $f_i$. The equilibrium is degenerate if $p_i \in \bd(f_i)$ but the term `stable equilibrium' refers to a nondegenerate one henceforth unless stated otherwise.
We denote the number of stable static equilibrium points of $P$ by $S(P)$ and we have $S(P) \leq F$.
\end{defn}

\begin{thm}\label{thm:dual}
Let $P$ be a  convex polytope {}. Then
\begin{equation}\label{main_dual}
    S(P)=\sum_{i=1}^F  \left \lfloor \frac{1}{2}+ \frac{\sum_{j=1}^F \sign\left((\mathbf{q}_j-\mathbf{q}_i)\mathbf{q}_j\right)}{2(F-1)}\right \rfloor.
\end{equation}
\end{thm}

Next we prove Theorems \ref{thm:main} and \ref{thm:dual}. In Subsection \ref{ss:dawson} we will show
that in case of homogeneous simplices the statement of Theorem \ref{thm:dual} follows from a result of Dawson \cite{Dawson}.

\subsection{Proof of Theorem \ref{thm:main}}

Let $p_i$ and $p_j$ denote two vertices and let $o$ denote the center of mass of the polytope $P$. It appears to be intuitively clear that, for any fixed value of $|\mathbf{r}_i|-|\mathbf{r}_j|$, if  the angle between the rays $op_1$ and $op_2$ is sufficiently small,  only one of the vertices $p_1,p_2$ may carry a static equilibrium. We express this observation by the following
\begin{defn}
Let $p_i, p_j$ be two vertices of $P$, with respective position vectors $\mathbf{r}_i, \mathbf{r}_j.$ We say that $p_i$ is \emph{in the shadow of} $p_j$  (or, alternatively, $p_j$ is \emph{shadowing} $p_i$) if the existence of $p_j$ prohibits $p_i$ to be a static equilibrium point.
\end{defn}
\noindent This intuition can be clarified by the following simple
\begin{lem}\label{lem:1}
Vertex $p_i$ is in the shadow of vertex $p_j$ if and only if
\begin{equation}\label{scalar}
(\mathbf{r}_i-\mathbf{r}_j)\mathbf{r}_i < 0.
\end{equation}
\end{lem}
\begin{proof}
Condition (\ref{scalar}) implies that the support plane of $P$ at $p_i$ will intersect $P$, so $p_i$ can not be an equilibrium point.
\end{proof}
\begin{cor}\label{cor:p_sh_inwards} 
If vertex $p_j$ shadows vertex $p_i$, then
\begin{equation}\label{ri_lt_rj}
\left|\mathbf{r}_i \right| < \left|\mathbf{r}_j \right|.
\end{equation}
\end{cor}
\begin{proof}
By Lemma~\ref{lem:1},
$\mathbf{r}_i^2 < \mathbf{r}_i \mathbf{r}_j,$
but
$0<\left(\mathbf{r}_i - \mathbf{r}_j \right)^2 =
\mathbf{r}_i^2 - 2\mathbf{r}_i\mathbf{r}_j + \mathbf{r}_j^2 <
-\mathbf{r}_i^2 + \mathbf{r}_j^2.
$
\end{proof}
\begin{lem}\label{lem:2}
If the vertex $p_i$ is not in the shadow of any other vertex $p_j$ $(j \not= i)$ then $P$ has an unstable equilibrium point at $p_i$.
\end{lem}
\begin{proof}
The support plane at $p_i$ defines two half-spaces.
If the vertex $p_i$ is not in the shadow of any other vertex then
the center of mass $o$ and all other vertices will be contained in the same half-space. This implies that the support plane of $P$ at $p_i$ does not intersect $P$, so $p_i$ is an unstable equilibrium point of $P$.
\end{proof}
\begin{rem}
There are three possible shadowing relationships between two vertices $p_i$ and $p_j$:
\begin{enumerate}
    \item Vertex $p_i$ is shadowing vertex $p_j$.
    \item Vertex $p_j$ is shadowing vertex $p_i$.
    \item Vertex $p_i$ and vertex $p_j$ are not shadowing each other.
\end{enumerate}
\end{rem}
\begin{defn}\label{def:shadowmatrix}
The elements $s_{i,j}$ of the $V \times V$ \emph{vertex shadowing matrix} $\mathbf{S}(P)$
are defined as follows: If $i\not=j$ then $s_{i,j}=-1$ if $p_j$ is shadowing $p_i$ and $s_{i,j}=1$ if $p_j$ is not shadowing $p_i$.
If $i=j$ then $s_{i,i}=0$. Based on Lemma \ref{lem:1} we have
\begin{equation}
    s_{i,j}=\sign \left((\mathbf{r}_i-\mathbf{r}_j)\mathbf{r}_i\right).
\end{equation}
\end{defn}

Now we conclude the proof of Theorem \ref{thm:main}.  
\begin{proof} Based on Definition \ref{def:shadowmatrix},
for the absolute value of the row-sums of the shadowing matrix  $\mathbf{S}(P)$ we can write
\begin{equation}
    \left |\sum_{j=1}^{V}s_{i,j}\right | \leq V-1
\end{equation}
with equality if and only if $p_i$ is an unstable static equilibrium point of $P$. According to the formula (\ref{main}) of Theorem \ref{thm:main} in this case we add one to the value of $U(P)$.
\end{proof}

\subsection{Proof of Theorem \ref{thm:dual}}

\begin{defn}
Let $f_i, f_j$ be two faces of $P$, with respective face vectors $\mathbf{q}_i, \mathbf{q}_j$ and let $o_i$ denote the projection of the center of mass $o$ onto the plane of the face $f_i$. We say that $f_i$ is \emph{in the shadow of} $f_j$  (or, alternatively, $f_j$ is \emph{shadowing} $f_i$) if the existence of $f_j$ prohibits $f_i$ to be a static equilibrium point.
\end{defn}
\begin{lem}\label{lem:1_dual}
Face $f_i$ is in the shadow of face $f_j$ if and only if

\begin{equation}\label{scalar_dual}
(\mathbf{q}_j-\mathbf{q}_i)\mathbf{q}_j < 0.
\end{equation}
\end{lem}

\begin{proof}
Condition (\ref{scalar_dual}) implies that the support plane of $P$ at $o_i$ will intersect $P$, so $o_i$ can not be an equilibrium point.
\end{proof}

\begin{cor}\label{cor:q_sh_outwards}
If face $f_j$ shadows face $f_j$, then
\begin{equation}\label{qj_lt_qi}
\left|\mathbf{q}_j \right| < \left|\mathbf{q}_i \right|.
\end{equation}
\end{cor}
\begin{proof}
By Lemma~\ref{lem:1_dual},
$\mathbf{q}_j^2 < \mathbf{q}_i \mathbf{q}_j,$
but
$0<\left(\mathbf{q}_j - \mathbf{q}_i \right)^2 =
\mathbf{q}_j^2 - 2\mathbf{q}_j\mathbf{q}_i + \mathbf{q}_i^2 <
-\mathbf{q}_j^2 + \mathbf{q}_i^2.
$
\end{proof}

\begin{lem}\label{lem:2_dual}
If the face $f_i$ is not in the shadow of any other face $f_j$ $(j \not= i)$ then $P$ has a stable equilibrium point on the face $f_i$.
\end{lem}
\begin{proof}
If $f_i$ is not in the shadow of any other face then $P$ will \emph{not} tip to any other face if positioned on $f_i$, so $f_i$ carries a stable equilibrium point.
\end{proof}

\begin{defn}\label{def:shadowmatrix_dual}
The elements $\bar s_{i,j}$ of the $F \times F$ \emph{face shadowing matrix} $\mathbf{\bar S}(P)$
are defined as follows: If $i\not=j$ then $\bar s_{i,j}=-1$ if $f_j$ is shadowing $f_i$ and $\bar s_{i,j}=1$ if $f_j$ is not shadowing $f_i$.
If $i=j$ then $\bar s_{i,i}=0$. Based on Lemma \ref{lem:1_dual} we have
\begin{equation}
    \bar s_{i,j}=\sign \left((\mathbf{q}_j-\mathbf{q}_i)\mathbf{q}_j\right).
\end{equation}
\end{defn}

Now we conclude the proof of Theorem \ref{thm:dual}.  
\begin{proof} Based on Definition \ref{def:shadowmatrix_dual},
for the absolute value of the row-sums of the shadowing matrix  $\mathbf{\bar S}(P)$ we can write
\begin{equation}
    \left |\sum_{j=1}^{F}s_{i,j}\right | \leq F-1
\end{equation}
with equality if and only if $f_i$ is a stable static equilibrium point of $P$. According to the formula (\ref{main_dual}) of Theorem \ref{thm:dual} in this case we add one to the value of $S(P)$.

\end{proof}

\section{The Algorithm}\label{s:alg}

\subsection{Problem statement} \label{ss:problem}
Formula (\ref{main}) of Theorem \ref{thm:main} is, by default, capable to compute the shadowing matrix $\mathbf{S}(P)$ of a given polytope $P$, based on position vectors $\mathbf{r}_i$, $(i = 1, \ldots, V)$.

However, one can also use this formula \emph{to find} monostatic polytopes with  $U(P)=1$  unstable equilibrium point.

\subsection{Notation and definition of equations}
\begin{defn}\label{def:firstvertex}
Let us denote the vertex at largest distance from the center of mass by $p_1$.
\end{defn}
\begin{lem}
Vertex $p_1$ always carries an unstable equilibrium.
\end{lem}
\begin{proof}
All vertices $p_i$ with $i>1$ are on the same side
of the support plane at $p_1$, so none of the vertices $p_i$, $i>1$
may shadow $p_1$. Via Lemma \ref{lem:1} this provides the proof.
\end{proof}
\begin{lem}\label{lem:mono}
Let $P$ be a mono-unstable polytope with $U(P)=1$. Then, for any $1 < i \leq V$ there exists at least one $0 < j \leq V$ such that $s_{i,j}=-1$.
\end{lem}
\begin{proof}
We prove Lemma \ref{lem:mono} by contradiction. If the statement of the Lemma is false then we have for some $i^{\star}>1$ a row of the shadowing matrix with positive entries (except for the main diagonal). This means that 
vertex $p_{i^{\star}}$ is not shadowed by any other vertex, i.e. it is an equilibrium (in addition to the equilibrium located at $p_1$). However, having $U(P)>1$ equilibria contradicts the statement of the lemma.
\end{proof}

\subsection{Material distribution}
Since polytopes have multi-level geometric structure, material homogeneity may be interpreted in various ways.
In subsection \ref{ss:gen} we introduced $h$-skeletons where material is uniform on the $h$-faces of the polyhderon.
Figure  \ref{fig:material}(a)-(c)
illustrate $h$-skeletons for 3-dimensional polyhedra.
While Theorems \ref{thm:main} and \ref{thm:dual} are valid irrespective of material distribution, if we want to utilize
these results to find monostatic polytopes (as mentioned in subsection \ref{ss:problem}), the difficulty of the resulting algebraic system will strongly depend on the material distribution. In this paper we concentrate on 0-skeletons which yield the simplest balance equations.
We also mention that computing for the $h=3$ homogeneous distribution in $d=3$ dimensions is identical to the problem of $0$-skeletons if we only investigate simplices;
however, the equilibria of that class of polyhedra are well understood. 
The dual case, corresponding to Theorem \ref{thm:dual} is more challenging: none of the $h=i$ ($i=0,1,\dots d$)  homogeneous skeletons yields convenient algebraic equations. Nevertheless, as we will point out in Subsection \ref{ss:dawson}, one may define a somewhat artificial distribution for which computations appear to be straightforward.

\subsection{Control parameters} 

The problem is defined by 3 integer parameters: $d$ is the dimension of the polytope,  $0 \leq h \leq  d$ defines the material distribution and $V > d$ provides the number of vertices. For brevity, we will characterize the problem with the triplet $(d,h,V)$. If $V$ is not specified, then the triplet will refer to the problem of finding the smallest value of $V$ for which a monostatic polyhedron exists.

\noindent Now we regard the following set of equations:
\begin{equation}\label{eq:main}
    s_{i,j}=-1, \quad i=2,3,\dots V, \quad j \in \{1,2,\dots i-1\}
\end{equation}
\noindent (note the condition $i<j$ that will be justified below) and proceed by
\begin{defn}\label{def:expand}
The $(d,h,V)$-\emph{expansion} of (\ref{eq:main}) is constructed in the following manner:
\begin{enumerate}
    \item We construct \emph{systems}, each  consisting of $V-1$ inequalities with index $i$ running from $2$ to $V$, defining $(V-1)!$ systems by using admissible permutations of $j$.
    \item  Each system is supplemented by $d$ moment equations
    guaranteeing that the center of mass $o$  of the $h$-homogeneous polytope is at the origin.
\end{enumerate}
So the $(d,h,V)$-expansion of (\ref{eq:main}) in $d$-dimension will consist of $(V-1)!$ systems, each having $V-1$ inequalities and $d$  equations to which, for brevity, we will henceforth refer as a \emph{system of inequalities} or \emph{system}.
\end{defn}

\subsection{Possible scenarios}
Solving the $(V-1)!$ systems of $V-1$ inequalities and $d$  equations for the $Vd$ scalar coordinates $r_{i, k}$, ($i=1,2, \dots V$, $k=1,2, \dots d$) of the position vectors $\mathbf{r}_i$ may yield two types of result:

\begin{enumerate}
\item  If \emph{none} of the $(V-1)!$ systems has a solution, that implies the existence of a row for $i>1$ in the shadowing matrix with \emph{only positive entries}. This is the necessary and sufficient condition for $P$ \emph{not} to be monostatic, so it proves the nonexistence of $d$-dimensional monostatic polytopes $P$ with $V$ vertices.

    \item If \emph{any} of the $(V-1)!$ systems  has a solution, that implies that for $i>1$, in \emph{each row} of the shadowing matrix there is a negative element. According to Lemma \ref{lem:mono}, this is the necessary and sufficient condition for $P$ to be monostatic. However, we remark that our algorithm does not include the condition that the solution should be convex,  so this has to be checked additionally.
     
\end{enumerate}

\begin{rem}
The algorithm is only testing the lower triangle of the shadowing matrix, characterized by $i>j$. It is easy to see that this is sufficient: If the $(d,h,V)$ problem has
a solution, this means that there exists a monostatic, $d$-dimensional, $h$-homogeneous polytope with $V$ vertices.  Based on Lemma \ref{lem:mono}, the shadowing matrix $\mathbf{S}(P)$ of a monostatic polytope $P$ has the property that for $i>1$, in each row there will be at least one negative element $s_{i,j}=-1$.
If all vertices $p_i (i=2 \ldots V)$ are labelled such that $\left|\mathbf{r}_1\right| \geq \left|\mathbf{r}_2\right| \geq \ldots \geq \left|\mathbf{r}_V\right| $, no negative entries in the upper triangle of $\mathbf{S}(P)$ can occur by Corollary~\ref{cor:p_sh_inwards}. In other words, $p_1$ always carries an unstable equilibrium, $p_2$ is shadowed by $p_1$, $p_3$ is shadowed either by $p_1$ or $p_2$, etc., which finally yields $(V-1)!$ possible shadowing systems.
\end{rem}

\subsection{Initial conditions}\label{ss:initial}
In $d$ dimensions a rigid body has $d$ translational degrees of freedom and $\binom{d}{2}$ rotational degrees of freedom.  By fixing the origin at $\mathbf{r}=0$ the translations are eliminated. In addition, we fix 
$n= \binom{d}{2} +1$ coordinates to eliminate the rotations and also to specify one element of a family of similar polytopes.

In $d=2$ dimensions we have $n=2$, so, in addition to fixing the origin, we fix the coordinates of one point. In $d=3$ dimensions we have $n=4$, so, in addition to fixing the origin, we fix all 3 coordinates of the first point and one additional coordinate of the second point. In $d=4$ dimensions we have $n=7$, so, in addition to fixing the origin, we fix  all 4 coordinates of the first point and 3 coordinates of the second point.

\subsection{General system of inequalities for 0-skeletons in 3 dimensions, i.e. the $(d,h,V)=(3,0,V)$ problem.}

Equation (\ref{eq:main}) is equivalent to a system
of inequalities which can be written as

\begin{equation}\label{s2}
\sum_{k=1}^{d}r_{i,k}^2 \leq \sum_{k=1}^{d}r_{i,k}r_{j,k}, \quad (i=2, \ldots,V; j=1,\ldots,i-1)   %\quad k=1,2, \ldots d) Sanyi törölte
\end{equation}
\noindent expressing that the $i$-th point is
shadowed by the $j$-th point.

If we consider the $h=0$ problem then the balance equations for the center of mass can be written as:

\begin{equation}\label{s3}
    \sum_{i=1}^Vr_{i, k}=0, \quad k=1,2,\dots d.
\end{equation}

We remark that in case of a simplex, the balance equations for the $h=0$ and $h=d$ cases coincide.
Using the considerations of subsection \ref{ss:initial},  we fix 
$n= \binom{3}{2} +1$=4 coordinates:
\begin{equation}\label{43}
\begin{array}{rcl}
    r_{1,1} & = & 1 \\
    r_{1,2} & = & 0 \\
    r_{1,3} & = & 0 \\
    r_{2,3} & = & 0,
\end{array}
\end{equation}
which, by Definition \ref{def:firstvertex} implies
\begin{equation}\label{unitbound}
    |r_{i,k}| <1, \quad i >1, \quad k=1,2, \dots d.
\end{equation}
Henceforth, it is our goal to prove the unsolvability the system (\ref{s2}-\ref{43}) for various values of $V$.

\section{Tetrahedron: the $(d,h,V)=(3,3,4)$ and $(d,h,V)=(3,0,4)$ problems}\label{sec:ex}

\begin{lem}
There exists no mono-unstable tetrahedron.
\end{lem}

\begin{proof}

\begin{rem}
The statement of the Lemma is equivalent to the claim that the $(3,0,4)$-expansion of (\ref{s2}) has no solution. We will prove
the latter.
\end{rem}

Next, using  Definition \ref{def:expand} we construct the
$(3,0,4)$-expansion of (\ref{s2})
in two steps.

In the first step,  after substituting the initial conditions (\ref{43}) we obtain the following $V-1=3$ groups of inequalities:

\begin{equation}\label{44}
r_{2,1}^2+r_{2,2}^2 \leq r_{2,1}
\end{equation}

\begin{equation}\label{45}\left\{
\begin{array}{rcl}
r_{3,1}^2+r_{3,2}^2+r_{3,3}^2 & \leq & r_{3,1} \\
r_{3,1}^2+r_{3,2}^2+r_{3,3}^2 & \leq & r_{2,1}r_{3,1}+r_{2,2}r_{3,2}
\end{array}
\right\}
\end{equation}

\begin{equation}\label{46}\left\{
\begin{array}{rcl}
r_{4,1}^2+r_{4,2}^2+r_{4,3}^2 & \leq & r_{4,1} \\
r_{4,1}^2+r_{4,2}^2+r_{4,3}^2 & \leq & r_{2,1}r_{4,1}+r_{2,2}r_{4,2} \\
r_{4,1}^2+r_{4,2}^2+r_{4,3}^2 & \leq & r_{3,1}r_{4,1}+r_{3,2}r_{4,2}+r_{3,3}r_{4,3}
\end{array}
\right\}
\end{equation}

In the second step, we have $1 \times 2 \times 3 = 6$ possibilities to pick
one inequality each from the three groups (\ref{44}-\ref{46})
and complementing these with the $d=3$ substitution of (\ref{s3}) yields 6 systems,
each consisting of 3 equations and 3 inequalities.
Here the moment balance equations for the $h=0$ and $h=3$ problem coincide since the tetrahedron is a simplex, so the above system describes both problems. Here follows the expanded system:

\begin{equation}\label{47}\left\{
\begin{array}{lrcl}
(a) & r_{2,1}^2+r_{2,2}^2  & \leq  & r_{2,1} \\
(b)  & r_{3,1}^2+r_{3,2}^2+r_{3,3}^2 & \leq & r_{3,1}  \\
(c)  & r_{4,1}^2+r_{4,2}^2+r_{4,3}^2 & \leq & r_{4,1} \\
(d) & 1 + r_{2,1}+r_{3,1} + r_{4,1} & = & 0 \\
(e) & r_{2,2}+r_{3,2}+r_{4,2} & = & 0 \\
(f) & r_{3,3}+r_{4,3} & = & 0
\end{array}
\right\}
\end{equation}

\begin{equation}\label{48}\left\{
\begin{array}{lrcl}
(a) & r_{2,1}^2+r_{2,2}^2  & \leq  & r_{2,1} \\
(b)  & r_{3,1}^2+r_{3,2}^2+r_{3,3}^2 & \leq & r_{3,1}  \\
(c)  & r_{4,1}^2+r_{4,2}^2+r_{4,3}^2 & \leq & r_{2,1}r_{4,1}+r_{2,2}r_{4,2} \\
(d) & 1 + r_{2,1}+r_{3,1} + r_{4,1} & = & 0 \\
(e) & r_{2,2}+r_{3,2}+r_{4,2} & = & 0 \\
(f) & r_{3,3}+r_{4,3} & = & 0
\end{array}
\right\}
\end{equation}

\begin{equation}\label{49}\left\{
\begin{array}{lrcl}
(a) & r_{2,1}^2+r_{2,2}^2  & \leq  & r_{2,1} \\
(b)  & r_{3,1}^2+r_{3,2}^2+r_{3,3}^2 & \leq & r_{3,1}  \\
(c)  & r_{4,1}^2+r_{4,2}^2+r_{4,3}^2 & \leq & r_{3,1}r_{4,1}+r_{3,2}r_{4,2}+r_{3,3}r_{4,3} \\
(d) & 1 + r_{2,1}+r_{3,1} + r_{4,1} & = & 0 \\
(e) & r_{2,2}+r_{3,2}+r_{4,2} & = & 0 \\
(f) & r_{3,3}+r_{4,3} & = & 0
\end{array}
\right\}
\end{equation}

\begin{equation}\label{50}\left\{
\begin{array}{lrcl}
(a) & r_{2,1}^2+r_{2,2}^2  & \leq  & r_{2,1} \\
(b)  & r_{3,1}^2+r_{3,2}^2+r_{3,3}^2 & \leq & r_{2,1}r_{3,1}+r_{2,2}r_{3,2}  \\
(c)  & r_{4,1}^2+r_{4,2}^2+r_{4,3}^2 & \leq & r_{4,1} \\
(d) & 1 + r_{2,1}+r_{3,1} + r_{4,1} & = & 0 \\
(e) & r_{2,2}+r_{3,2}+r_{4,2} & = & 0 \\
(f) & r_{3,3}+r_{4,3} & = & 0
\end{array}
\right\}
\end{equation}

\begin{equation}\label{51}\left\{
\begin{array}{lrcl}
(a) & r_{2,1}^2+r_{2,2}^2  & \leq  & r_{2,1} \\
(b)  & r_{3,1}^2+r_{3,2}^2+r_{3,3}^2 & \leq & r_{2,1}r_{3,1}+r_{2,2}r_{3,2}  \\
(c)  & r_{4,1}^2+r_{4,2}^2+r_{4,3}^2 & \leq & r_{2,1}r_{4,1}+r_{2,2}r_{4,2} \\
(d) & 1 + r_{2,1}+r_{3,1} + r_{4,1} & = & 0 \\
(e) & r_{2,2}+r_{3,2}+r_{4,2} & = & 0 \\
(f) & r_{3,3}+r_{4,3} & = & 0
\end{array}
\right\}
\end{equation}

\begin{equation}\label{52}\left\{
\begin{array}{lrcl}
(a) & r_{2,1}^2+r_{2,2}^2  & \leq  & r_{2,1} \\
(b)  & r_{3,1}^2+r_{3,2}^2+r_{3,3}^2 & \leq & r_{2,1}r_{3,1}+r_{2,2}r_{3,2}  \\
(c)  & r_{4,1}^2+r_{4,2}^2+r_{4,3}^2 & \leq & r_{3,1}r_{4,1}+r_{3,2}r_{4,2}+r_{3,3}r_{4,3} \\
(d) & 1 + r_{2,1}+r_{3,1} + r_{4,1} & = & 0 \\
(e) & r_{2,2}+r_{3,2}+r_{4,2} & = & 0 \\
(f) & r_{3,3}+r_{4,3} & = & 0
\end{array}
\right\}
\end{equation}

Next we show that none of the systems (\ref{47})-(\ref{52}) has a solution.

\subsubsection{Systems (\ref{47}-\ref{49})}
From part (a) of either of (\ref{47}-\ref{49}) we have $r_{2,1}\geq 0$, and so (d) with (\ref{unitbound}) imply both $r_{3,1}\leq 0, r_{4,1}\leq 0$. This, however, contradicts (b), so systems (\ref{47}-\ref{49}) have no solution.

\subsubsection{System (\ref{50})}
The proof is essentially identical to the proof for systems (\ref{47}-\ref{49}) but the contradiction is found at (c) instead of (b).

\subsubsection{System (\ref{51})}

From (\ref{51})(a)  we have $r_{2,1}\geq 0$, and this implies, via
(\ref{51})(d) that $r_{3,1} <0, r_{4,1}<0 $. Substituting this into (\ref{51})(b) and (c) we get $r_{2,2}r_{3,2} >0, r_{2,2}r_{4,2}>0$, respectively implying that $r_{2,2}, r_{3,2}$ and $r_{4,2}$ have the same sign. This, however, contradicts (\ref{51})(e), so system (\ref{51}) has no solution.

\subsubsection{System (\ref{52})}
The proof is similar to the proof of system (\ref{51}). From (\ref{52})(a)  we have $r_{2,1}\geq 0$, and this implies, via
(\ref{52})(d) that $r_{3,1} <0, r_{4,1}<0 $. 
 This yields, via
(\ref{52})(b)  $r_{2,2}r_{3,2} >0 $ and further, via (\ref{52})(e) $r_{3,2}r_{4,2} <0$.  (\ref{52})(f) yields $r_{3,3}r_{4,3} <0 $. Since we already showed that
$r_{2,1}r_{4,1} <0$, substituting into (\ref{52})(c) yields a contradiction, so system (\ref{52}) has no solution.

\end{proof}

\section{Computing the lower bound for $V^U_0$: the $(d,h,V)=(3,0,\{5,6,7\})$ problems.}\label{sec:No-mono}

Our  goal is to prove  Theorem \ref{thm:7}, to which we now give an equivalent, more detailed formulations as
\begin{cor}\label{cor:7}
The $(d,h,V)=(3,0,\{5,6,7\})$ problems have no solution.  
\end{cor} We prove Corollary \ref{cor:7} (and thus also  Theorem \ref{thm:7}) and Lemma \ref{lem:main}) by applying the algorithm laid out in Section \ref{s:alg}.

\begin{proof}

We show that the system (\ref{s2})-(\ref{s3}) has no solution for $d=3, h=0$, $V \leq 7.$
The case of $V = 4$ was proven by elementary considerations in Section \ref{sec:ex}. The case $V = 5$ can also be proven in a similar way, however, only 7 of the 24 systems have such short proofs, the others require more steps: the total length is 22 (hand written) pages. The reason of that it is not included in this paper is that we have found another, optimization-based method which proves the unsolvability for $V = 6$ and $V = 7$, too. 
    
Departing from (\ref{s3}), $r_{V, k}$ can be expressed as
\begin{equation}\label{rVk-espressed}
    r_{V, k} = - \sum_{i=1}^{V-1}r_{i, k}, \quad k=1,2,3.
\end{equation}
The substitutions (\ref{rVk-espressed}) convert the system of inequalities (\ref{s2}) and equations (\ref{s3}) into $V-1$ inequalities:
  
\begin{equation}\label{inequalities2-V-1}
\sum_{k=1}^{3}r_{i,k}^2 - \sum_{k=1}^{3}r_{i,k}r_{j,k} \leq 0 , \quad i=2, \ldots,V-1,
\end{equation}
\begin{equation}\label{inequalityV}
\sum_{k=1}^{3}{\left( - \sum_{i=1}^{V-1}r_{i, k} \right)}^2 - \sum_{k=1}^{3}{\left( - \sum_{i=1}^{V-1}r_{i, k} \right)}r_{j,k} \leq 0 , 
\end{equation}
where $ j \in \{1,\ldots,i-1\}$, resulting in $(V-1)!$ systems.
   
According to (\ref{43}), $ r_{1,1}  =  1, r_{1,2} , r_{1,3} , r_{2,3}  =  0$,  and since $r_{V,1}$, $r_{V,2}$ and $r_{V,3}$ were eliminated in (\ref{rVk-espressed}), 
the number of free variables $r_{i,k}$ is $3V-7$.
Let us define the parametric function $f: \mathbb{R}^{3V-7} \rightarrow \mathbb{R}$ 
as the weighted sum of the left hand sides of the inequalities (\ref{inequalities2-V-1})-(\ref{inequalityV}) above
\begin{equation}\label{function-f}
  f = \sum_{i=2}^{V-1} c_i \left( \sum_{k=1}^{3}r_{i,k}^2 - \sum_{k=1}^{3}r_{i,k}r_{j,k} \right) 
    +c_V \left[ \sum_{k=1}^{3}{\left( - \sum_{i=1}^{V-1}r_{i, k} \right)}^2 - \sum_{k=1}^{3}{\left( - \sum_{i=1}^{V-1}r_{i, k} \right)}r_{j,k} \right] 
\end{equation},
where the coefficients ${c_i}, i=2,\ldots,V$ are arbitrary positive numbers.

Let us fix the values $ j \in \{1,\ldots,i-1\}$ arbitrarily. If there exist positive
coefficients ${c_i}, i=2,\ldots,V$ such that $f$ is positive everywhere, then 
the unsolvability of system (\ref{inequalities2-V-1})-(\ref{inequalityV}) follows: 
assume for contradiction that it has a solution, and substitute this solution in (\ref{function-f}), yielding a non-positive value of $f$, which is not possible once $f$ is positive.

How to find coefficients and how to check the positivity of $f$?
Positive polynomials \cite{PrestelDelzell2001,Marshall2008} have many applications \cite{Lasserre2010,HenrionGarulli2005}.
  
However, in our case, suprisingly enough a simple
randomized search for the coefficients works, at least up to $V=7$ vertices. We benefit from that if $f$ is convex, then its minimum value can be calculated simply.  The positivity of $f$ can be proved by showing its convexity and finding a positive minimum. Since $f$ is a multi-variate polynomial of degree two, the first-order conditions of minimality yield a system of linear equations. In order to check convexity, the second-order condition is the positive definiteness of $Hf$, the Hessian of $f$.
Note that $Hf$ depends on coefficients ${c_i}$ only, following again from that $f$ is quadratic.

The algorithm for proving unsolvability of the system of inequalities (\ref{inequalities2-V-1})-(\ref{inequalityV}) is summarized below.

\begin{verbatim}
for all systems of inequalities 
step 1. generate random positive integer values of c_2,...,c_n
step 2. if Hf is positive definite 
        then if min f > 0, 
             then print('this system is unsolvable')
                  go to the next system in the for cycle
             else go to step 1  
        else go to step 1 
end for
\end{verbatim}
   
Appendices 1-4 include the coefficients and the minimum vale of function $f$ for all systems written for $V=4,5,6,7$ vertices, respectively. Since the minimum values are positive, all systems are unsolvable.
    
The computational approach above does not include numerical errors because all calculations deal with integer and rational numbers, resulting in exact rational numbers, too. 

\end{proof}

\begin{rem}

In case of $V=8$ vertices, there are some systems of inequalities (e.g. system (\ref{inequalities2-V-1})-(\ref{inequalityV}) written for the choice of $j=i-1$ for all $i=2,\ldots,V-1$ ), where we could not prove that $f$ is positive with appropriate coefficients as all the random trials
led to negative minima, and at the same time at least one inequality of the system (\ref{inequalities2-V-1})-(\ref{inequalityV}) was violated -- neither unsolvability of all of these systems, or solvability of at least one of these systems follows.

The lower bound for the number of minimal vertices, that a three-dimensional  mono-unstable 0-skeleton must have, has now been improved to 8. Upper bounds may come when, for a given $V$, a system of polynomial inequalities (\ref{inequalities2-V-1})-(\ref{inequalityV}) has a solution, and the  corresponding polyhedron is convex (see Subsection 3.5). Finding a solution of polynomial inequalities is itself a challenging problem \cite{GrigorievVorobjov1988,Strzebonski2000}.
\end{rem}

\section{Concluding remarks}\label{sec:con}

\subsection{The dual problem: search for mono-stable polyhedra}\label{ss:dawson}

So far we demonstrated how Theorem \ref{thm:main} can be converted into an algorithm to establish the lower bound
for the number of vertices of a mono-unstable polyhedron.  To illustrate the algorithm we computed the case of 0-skeletons
and found the lower bound $V^U_0 \geq 8$ (Theorem \ref{thm:7}).   % BS inserted "\qeg" instead of "=" , and deleted "to be"

 Theorem \ref{thm:dual} is  the dual of Theorem \ref{thm:main}, and it could, in principle, serve as the basis of a dual algorithm, searching for mono-stable polyhedra. However, as we will explain below, the only mass-distribution where this computation would be of comparable difficulty as 0-skeletons for mono-unstable polyhedra is physically rather counter-intuitive. To better understand the background, we point out the connection of our work to Dawson's research on mono-stable simplices.

Starting on a proof by Conway for the non-existence of a mono-stable tetrahedron,  Dawson \cite{Dawson} investigated the existence of mono-stable simplices in $d$-dimensions.
 This research was continued in \cite{Dawson2,Dawson3} and ultimately led to the proof that for $d<9$ no mono-stable simplex exists and for $d>9$ there exist mono-stable simplices. The $d=9$ case is not yet resolved.
 The problem of the existence of mono-stable simplices  is closely related to our problem and below we will point out the main connections as well as the main differences.
 
 The key idea in Dawson's arguments is what he calls the \emph{projection criterion}:
 \begin{equation}\label{dawson1}
\lvert\mathbf{{x}_i}\rvert < \lvert\mathbf{{x}_j}\rvert\cos\theta_{ij},
\end{equation}
where $\mathbf{x}_i$ is orthogonal to face $f_i$ of the simplex $s$
and $|\mathbf{x}_i|$ measures
the area of $f_i$ and $\theta_{ij}$ is the angle between $\mathbf{x}_i$ and $\mathbf{x}_j$.
If the projection criterion (\ref{dawson1}) holds, then a $h=3$- (or $h=0$)-homogeneous simplex, if placed on face $f_i$, will tip over to face $f_j$, so we will call (\ref{dawson1}) the \emph{tipping condition}. Next we show

\begin{prop}\label{prop:tip}
In the case of homogeneous simplices and 0-skeletal simplices,  equation (\ref{scalar_dual}) is equivalent to Dawson's tipping condition (\ref{dawson1}).
\end{prop}

\begin{proof}
First we show that

\begin{equation}\label{forditott_index}
\lvert\mathbf{x}_i\rvert\lvert\mathbf{q}_i\rvert = d \cdot\frac{Vol(s)}{d+1}, 
\end{equation}
where $Vol(s)$ denotes the volume of the simplex $s$.

%2 és 3 dimenzióban megvan, d dimenzióban?
%d-dimenziós szimplex esetén a súlypont a súlyvonalat 1:d arányban osztja.
%In a d-dimensional simplex the barycenter divides the median in 1:d.
Let $s$ be a $d$-dimensional simplex with center of mass $o$, let $f_i$ be a $(d-1)$-dimensional face of $s$, let the center of mass of $f_i$ be denoted by $o_i$ and let $v_i$ be the vertex opposite $f_i$. It is known \cite{KrantzMcCarthyParks} that $|v_i,o|/|o,o_i|=d$ (where $|a,b|$ denotes the length of the line segment $\overline{ab}$). 

Let $s_i$ be a $d$-dimensional simplex defined by the following $d+1$ vertices: face $f_i$ defines $d$ vertices and we add $o$ as the $(d+1)$st vertex. $\lvert\mathbf{q}_i\rvert$ is the height of $s_i$ orthogonal to $f_i$. From this it follows that $Vol(s)/Vol(s_i)=d+1$ for $i=1,2, \dots d+1$.
Since the left hand side of (\ref{forditott_index}) is constant, (\ref{dawson1}) implies that

\begin{equation}
    \lvert\mathbf{q}_j\rvert - \lvert\mathbf{q}_i\rvert\cos\theta_{ij} < 0,
\end{equation}
and this yields (\ref{scalar_dual}) via
\begin{equation}
\lvert\mathbf{q}_j\rvert^2 - \lvert\mathbf{q}_j\rvert\lvert\mathbf{q}_i\rvert\cos\theta_{ij} < 0.
\end{equation}

\end{proof}

As noted in Proposition \ref{prop:tip}, Dawson's tipping condition (\ref{dawson1}) only applies if 
the polyhedron is either a homogeneous simplex or a 0-skeleton  of a simplex and
it is only equivalent to equation (\ref{scalar_dual}) under the same condition.
Now we show that there exists a material distribution for which
Dawson's tipping condition is \emph{reversed}: if (\ref{dawson1}) holds then the simplex will tip from
face $f_j$ to face $f_i$.

\begin{defn}\label{def:fheavy}
We will call a $3$-dimensional polytope with $F$ faces a \emph{dual 0-skeleton} if $\sum \mathbf{q}_i=0$,
with the vectors  $\mathbf{q}_i$, ($i=1,2, \dots F$)  defined in subsection \ref{ss:problemstatement}.
\end{defn}
The balance equations for the center of mass of a dual 0-skeleton can be written as
\begin{equation}\label{s3dual}
    \sum_{i=1}^{F}q_{i, k}=0, \quad k=1,2,\dots d,
\end{equation}
which is analogous to equation (\ref{s3}). Based on Theorem \ref{thm:dual} and  the balance equations (\ref{s3dual}) we can construct the dual version of the algorithm presented in section \ref{s:alg}.

The name for dual 0-skeletons is motivated by their role in case of simplices where
the dual 0-skeleton is uniquely defined and corresponds to a center of mass in the interior
of the simplex. In fact, it is straightforward to find the center of mass $o$ for the dual 0-skeleton in a simplex: 
let $S$ be a $d$-dimensional simplex and we regard the vectors $\mathbf{x}_i$ ($i=0,1, \dots d+1)$ introduced in equation (\ref{dawson1}). Assume all $\mathbf{x}_i$ have identical origin $o$ and we denote the endpoints of the vectors by $X_i$. Now we regard the set of planes $p_i$ each of which is normal to the corresponding vector $\mathbf{x}_i$ at the endpoint $X_i$. It is easy to see that the simplex $S'$ defined by these planes will be similar to $S$, the set of vectors $\mathbf{q'}_i$ for $S'$ can be defined as $\mathbf{q'}_i=\mathbf{x}_i$. If we place unit masses at the points $X_i$ then this mass distribution defines the dual 0-skeleton of $S'$.
Dawson proved that a $h=3$ homogeneous (or $h=0$ homogeneous 0-skeleton) simplex can  tip from face $i$ to face $j$ if and only if his tipping condition (\ref{dawson1})
is true. In case of a dual 0-skeleton simplex Dawson's tipping condition is \emph{reversed}: it implies the opposite, i.e. that it would tip from face $j$ to face $i$.

\begin{prop} \label{prop:counter-tip}
For the dual 0-skeleton of any $d$-dimensional simplex $S$ the tipping condition (\ref{dawson1}) is reversed.
\end{prop}

\begin{proof}
The construction scheme described at Definition~\ref{def:fheavy} proves that the center of mass $o$ for the dual-0-skeleton of a $d$-dimensional simplex $S$ is an interior point of $S$ and $\mathbf{q}_i = \alpha \mathbf{x}_i$ with $\alpha>0$ holds for any pair $( \mathbf{x}_i, \mathbf{q}_i)$. Expressing the face shadowing condition (\ref{scalar_dual}) in terms of $\mathbf{x}_i$ yields $(\mathbf{x}_j-\mathbf{x}_i)\mathbf{x}_j < 0$, which can be obtained from (\ref{dawson1}) by interchanging subscripts $i$ and $j$.
\end{proof}

In case of $d$-dimensional polytopes with $F > d+1$ faces the center of mass corresponding to a dual 0-skeleton  may not be in the interior of the polytope.  Figure \ref{fig:dualskeleton} illustrates the 0-skeleton and the dual 0-skeleton of a triangle and also shows a quadrangle where the center of mass for the dual skeleton is not contained in the interior.

\begin{figure}[ht]
\begin{center}
\includegraphics[width=\columnwidth]{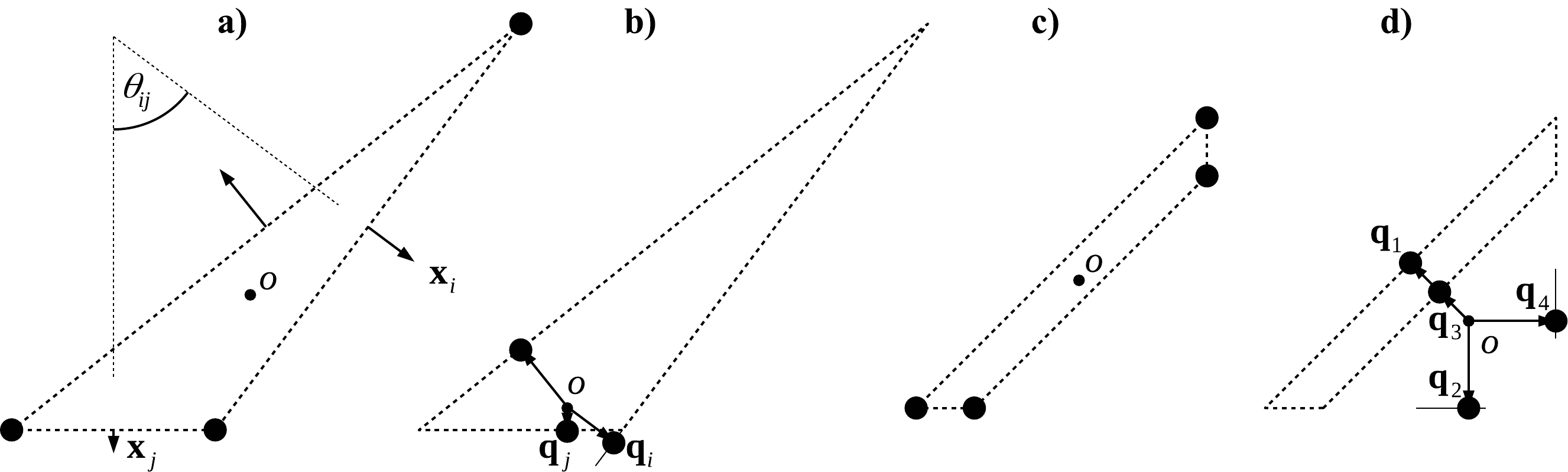}
\caption{Illustration of 0-skeletons and  dual 0-skeletons in 2D. (a) 0-skeleton of triangle,
(b) uniquely defined dual 0-skeleton of the same triangle (c) 0-skeleton of a quadrangle (d) possible dual 0-skeleton of the same quadrangle.  Large black dots mark unit masses.}
\label{fig:dualskeleton}
\end{center}
\end{figure}

Summarizing, we may say that the utilization of Theorem \ref{thm:dual} as the basis of an algorithm to find
mono-stable  polyhedra is computationally feasible only if we investigate dual 0-skeletons, however, the latter
 (except for the case of simplices)  may not be physically relevant, so the 
dual algorithm, based on Theorem \ref{thm:dual} does not appear to be of practical interest. Since the attention was previously focused on mono-stable polyhedra (and the mono-unstable case was not considered), this observation may explain why this algorithm was not investigated earlier.

\subsection{Summary}
In this paper we improved the previously known \cite{Conway, balancing} lower bound $V^U_0 \geq 5$ on the minimal number of vertices for a convex, mono-unstable 0-skeleton to $V^U_0 \geq 8$. This result also implies the same lower bound for
the minimal number of vertices for a convex mono-monostatic 0-skeleton, so we also proved $V^*_0 \geq 8$.

On one hand, we think that these lower bounds are not yet close to the actual minimal values. On the other hand, although no mathematical evidence exists, intuitively it looks plausible that the same lower bounds are valid for $V^U_3, V^*_3$, i.e., for homogeneous polyhedra.

The algorithm presented in this paper is, in principle, also capable to compute  $V^U_3$, however, we did not yet attempt to implement the balance equations for this case. Also, the same algorithm is (again, in principle) capable to compute higher dimensional problems. Dawson \cite{Dawson} investigated the minimal dimension in which a simplex may be mono-stable
and in our notation his result can be written for $d=10$ dimensions as  $V^S_0=11$.

\include{appendix1}

\include{appendix2}

\include{appendix3}

\include{appendix4}

\end{document}

%% file: appendix1.tex
\section*{Appendix 1: $V=4$}

\begin{center}
\begin{tabular}{|c|c|c|c|c|c|c}
\hline
 $i=2$ & $i=3$ &        &       &       &             \\   
  $j$  &  $j$  &  $c_2$ & $c_3$ & $c_4$ &  $\min f$   \\
 \hline
 1 & 1 & 94 & 46 & 97 & 1560613/17904\\
\hline
1 & 2 & 45 & 101 & 39 & 389290/18397\\
\hline
1 & 3 & 70 & 28 & 97 & 696052/63615\\
\hline
2 & 1 & 18 & 91 & 35 & 1732/417\\
\hline
2 & 2 & 19 & 53 & 57 & 79192/11149\\
\hline
2 & 3 & 44 & 84 & 26 & 137704/9511\\
\hline
\end{tabular}
\end{center}

%% file: appendix2.tex
\section*{Appendix 2: $V=5$}
\begin{center}
\begin{tabular}{|c|c|c|c|c|c|c|c|c}
\hline
 $i=2$ & $i=3$ & $i=4$ &       &       &       &       &             \\   
  $j$  &  $j$  &  $j$  & $c_2$ & $c_3$ & $c_4$ & $c_5$ &  $\min f$   \\
 \hline
 1 & 1 & 1 & 99 & 60 & 45 & 101 & 3607295/47108\\
\hline
1 & 1 & 2 & 39 & 70 & 28 & 97 & 1181819/48726\\
\hline
1 & 1 & 3 & 18 & 91 & 35 & 19 & 266291/25140\\
\hline
1 & 1 & 4 & 53 & 57 & 44 & 84 & 4575347/142162\\
\hline
1 & 2 & 1 & 94 & 15 & 51 & 48 & 12918697/714980\\
\hline
1 & 2 & 2 & 71 & 78 & 40 & 93 & 15857847/752306\\
\hline
1 & 2 & 3 & 72 & 68 & 79 & 52 & 40813288/1367773\\
\hline
1 & 2 & 4 & 50 & 68 & 75 & 91 & 83316661/7629667\\
\hline
1 & 3 & 1 & 75 & 28 & 99 & 73 & 362560/428323\\
\hline
1 & 3 & 2 & 11 & 16 & 65 & 10 & 10133/1524\\
\hline
1 & 3 & 3 & 13 & 53 & 26 & 73 & 3251269/232212\\
\hline
1 & 3 & 4 & 41 & 18 & 71 & 53 & 998744/1696685\\
\hline
2 & 1 & 1 & 82 & 88 & 35 & 86 & 62954917/1415460\\
\hline
2 & 1 & 2 & 101 & 71 & 37 & 37 & 2228629/95016\\
\hline
2 & 1 & 3 & 62 & 48 & 29 & 76 & 6918037/658372\\
\hline
2 & 1 & 4 & 4 & 6 & 12 & 38 & 2497/2840\\
\hline
2 & 2 & 1 & 41 & 59 & 94 & 99 & 103518673/36049948\\
\hline
2 & 2 & 2 & 12 & 31 & 7 & 33 & 482793/145192\\
\hline
2 & 2 & 3 & 12 & 24 & 42 & 53 & 109287/45176\\
\hline
2 & 2 & 4 & 59 & 68 & 82 & 66 & 122706853/8227804\\
\hline
2 & 3 & 1 & 85 & 97 & 95 & 32 & 640363741/27814368\\
\hline
2 & 3 & 2 & 82 & 32 & 55 & 68 & 5263121/4336998\\
\hline
2 & 3 & 3 & 60 & 67 & 13 & 50 & 51850/32607\\
\hline
2 & 3 & 4 & 41 & 72 & 75 & 50 & 5770729/750772\\
\hline
\end{tabular}
\end{center}

%% file: appendix3.tex
\section*{Appendix 3: $V=6$}

\begin{tabular}{|c|c|c|c|c|c|c|c|c|c|c}
\hline
 $i=2$ & $i=3$ & $i=4$ & $i=5$ &        &       &       &       &       &            \\   
  $j$  &  $j$  &  $j$  &  $j$  &  $c_2$ & $c_3$ & $c_4$ & $c_5$ & $c_6$ & $\min f$   \\
 \hline
 1 & 1 & 1 & 1 & 60 & 45 & 101 & 39 & 70 & 34831755/587084\\
\hline
1 & 1 & 1 & 2 & 72 & 68 & 79 & 52 & 54 & 209337151/5077748\\
\hline
1 & 1 & 1 & 3 & 91 & 75 & 28 & 99 & 73 & 4127160151/131553322\\
\hline
1 & 1 & 1 & 4 & 45 & 14 & 20 & 11 & 16 & 1157095/171946\\
\hline
1 & 1 & 1 & 5 & 65 & 10 & 13 & 53 & 26 & 1447/450\\
\hline
1 & 1 & 2 & 1 & 73 & 91 & 19 & 44 & 56 & 13991759/757344\\
\hline
\end{tabular}
\newpage
\begin{tabular}{|c|c|c|c|c|c|c|c|c|c|c}
\hline
 $i=2$ & $i=3$ & $i=4$ & $i=5$ &        &       &       &       &       &            \\   
  $j$  &  $j$  &  $j$  &  $j$  &  $c_2$ & $c_3$ & $c_4$ & $c_5$ & $c_6$ & $\min f$   \\
\hline
1 & 2 & 4 & 1 & 39 & 24 & 56 & 67 & 67 & 17833691/3401314\\
\hline
1 & 2 & 4 & 2 & 95 & 29 & 85 & 90 & 67 & 1502064923/97171977\\
\hline
1 & 2 & 4 & 3 & 38 & 56 & 42 & 83 & 56 & 1313617/931200\\
\hline
1 & 2 & 4 & 4 & 83 & 36 & 81 & 46 & 84 & 75788723/15973916\\
\hline
1 & 2 & 4 & 5 & 40 & 56 & 98 & 99 & 92 & 6347510/2334201\\
\hline
1 & 3 & 1 & 1 & 49 & 24 & 26 & 74 & 63 & 8673708/3219559\\
\hline
\end{tabular}
\newpage
\begin{tabular}{|c|c|c|c|c|c|c|c|c|c|c}
\hline
 $i=2$ & $i=3$ & $i=4$ & $i=5$ &        &       &       &       &       &            \\   
  $j$  &  $j$  &  $j$  &  $j$  &  $c_2$ & $c_3$ & $c_4$ & $c_5$ & $c_6$ & $\min f$   \\
\hline
2 & 1 & 3 & 3 & 58 & 81 & 37 & 46 & 30 & 713485811/125927456\\
\hline
2 & 1 & 3 & 4 & 36 & 66 & 66 & 27 & 81 & 4196136/751361\\
\hline
2 & 1 & 3 & 5 & 41 & 85 & 39 & 101 & 14 & 1116420809/351084428\\
\hline
2 & 1 & 4 & 1 & 53 & 33 & 94 & 82 & 20 & 1713426815/123458649\\
\hline
2 & 1 & 4 & 2 & 49 & 75 & 20 & 36 & 75 & 762159/69086\\
\hline
2 & 1 & 4 & 3 & 55 & 69 & 67 & 74 & 57 & 1042733481/92310944\\
\hline
\end{tabular}
\newpage
\begin{tabular}{|c|c|c|c|c|c|c|c|c|c|c}
\hline
 $i=2$ & $i=3$ & $i=4$ & $i=5$ &        &       &       &       &       &            \\   
  $j$  &  $j$  &  $j$  &  $j$  &  $c_2$ & $c_3$ & $c_4$ & $c_5$ & $c_6$ & $\min f$   \\
\hline
2 & 3 & 2 & 5 & 38 & 47 & 20 & 78 & 48 & 21361464/3421919\\
\hline
2 & 3 & 3 & 1 & 36 & 68 & 56 & 42 & 40 & 35987/10525\\
\hline
2 & 3 & 3 & 2 & 34 & 94 & 87 & 25 & 19 & 259637031/42854104\\
\hline
2 & 3 & 3 & 3 & 66 & 75 & 32 & 46 & 66 & 823988/2585579\\
\hline
2 & 3 & 3 & 4 & 38 & 64 & 37 & 76 & 33 & 242502073/129859746\\
\hline
2 & 3 & 3 & 5 & 13 & 43 & 54 & 61 & 70 & 103829032/113344853\\
\hline
2 & 3 & 4 & 1 & 13 & 40 & 55 & 36 & 8 & 20185411/6745636\\
\hline
2 & 3 & 4 & 2 & 24 & 64 & 50 & 94 & 35 & 233627303/62970814\\
\hline
2 & 3 & 4 & 3 & 37 & 92 & 98 & 10 & 17 & 92066919/64755854\\
\hline
2 & 3 & 4 & 4 & 19 & 47 & 36 & 25 & 65 & 272971/7199827\\
\hline
2 & 3 & 4 & 5 & 18 & 72 & 65 & 64 & 21 & 286613388/74974979\\
\hline
\end{tabular}

%% file: appendix4.tex
\section*{Appendix 4: $V=7$}

\hspace*{-20mm}
\begin{tabular}{|c|c|c|c|c|c|c|c|c|c|c|c|c}
\hline
 $i=2$ & $i=3$ & $i=4$ & $i=5$ & $i=6$ &        &       &       &       &       &       &            \\   
  $j$  &  $j$  &  $j$  &  $j$  &  $j$  &  $c_2$ & $c_3$ & $c_4$ & $c_5$ & $c_6$ & $c_7$ & $\min f$   \\
 \hline
1 & 1 & 1 & 1 & 1 & 45 & 101 & 39 & 70 & 28 & 97 & 109650545/2706833\\
\hline
1 & 1 & 1 & 1 & 2 & 44 & 84 & 26 & 91 & 94 & 61 & 3036840881/239998244\\
\hline
1 & 1 & 1 & 1 & 3 & 71 & 78 & 40 & 93 & 72 & 68 & 1989097815/47376769\\
\hline
1 & 1 & 1 & 1 & 4 & 79 & 52 & 54 & 13 & 20 & 64 & 10276837/972622\\
\hline
1 & 1 & 1 & 1 & 5 & 84 & 64 & 101 & 71 & 37 & 37 & 7841750065/317391094\\
\hline
1 & 1 & 1 & 1 & 6 & 44 & 54 & 62 & 48 & 29 & 76 & 411777625/35419432\\
\hline
\end{tabular}
\newpage
\hspace*{-20mm}
\begin{tabular}{|c|c|c|c|c|c|c|c|c|c|c|c|c}
\hline
 $i=2$ & $i=3$ & $i=4$ & $i=5$ & $i=6$ &        &       &       &       &       &       &            \\   
  $j$  &  $j$  &  $j$  &  $j$  &  $j$  &  $c_2$ & $c_3$ & $c_4$ & $c_5$ & $c_6$ & $c_7$ & $\min f$   \\
 \hline
1 & 1 & 2 & 1 & 6 & 59 & 28 & 42 & 83 & 28 & 99 & 8024303873/762217686\\
\hline
1 & 1 & 2 & 2 & 1 & 56 & 83 & 52 & 53 & 24 & 45 & 4257193309/674507697\\
\hline
1 & 1 & 2 & 2 & 2 & 81 & 24 & 97 & 96 & 86 & 35 & 16065156349/2016826268\\
\hline
1 & 1 & 2 & 2 & 3 & 61 & 91 & 40 & 77 & 61 & 19 & 63887717789/12360072330\\
\hline
1 & 1 & 2 & 2 & 4 & 82 & 83 & 56 & 71 & 46 & 64 & 99677464683/7159335868\\
\hline
1 & 1 & 2 & 2 & 5 & 53 & 68 & 40 & 87 & 55 & 56 & 299300717/54665121\\
\hline
\end{tabular}
\newpage
\hspace*{-20mm}
\begin{tabular}{|c|c|c|c|c|c|c|c|c|c|c|c|c}
\hline
 $i=2$ & $i=3$ & $i=4$ & $i=5$ & $i=6$ &        &       &       &       &       &       &            \\   
  $j$  &  $j$  &  $j$  &  $j$  &  $j$  &  $c_2$ & $c_3$ & $c_4$ & $c_5$ & $c_6$ & $c_7$ & $\min f$   \\
\hline
1 & 1 & 3 & 3 & 1 & 47 & 36 & 25 & 65 & 44 & 88 & 218142/26717755\\
\hline
1 & 1 & 3 & 3 & 2 & 24 & 31 & 18 & 72 & 65 & 64 & 9885533/826514\\
\hline
1 & 1 & 3 & 3 & 3 & 21 & 28 & 5 & 53 & 79 & 27 & 10013379/1744298\\
\hline
1 & 1 & 3 & 3 & 4 & 9 & 45 & 33 & 29 & 24 & 63 & 568056711/78479332\\
\hline
1 & 1 & 3 & 3 & 5 & 38 & 54 & 77 & 89 & 71 & 69 & 1990534206055/445292803532\\
\hline
1 & 1 & 3 & 3 & 6 & 70 & 71 & 27 & 97 & 63 & 50 & 101395336099/13853474788\\
\hline
\end{tabular}
\newpage
\hspace*{-20mm}
\begin{tabular}{|c|c|c|c|c|c|c|c|c|c|c|c|c}
\hline
 $i=2$ & $i=3$ & $i=4$ & $i=5$ & $i=6$ &        &       &       &       &       &       &            \\   
  $j$  &  $j$  &  $j$  &  $j$  &  $j$  &  $c_2$ & $c_3$ & $c_4$ & $c_5$ & $c_6$ & $c_7$ & $\min f$   \\
\hline
1 & 1 & 4 & 4 & 2 & 33 & 62 & 74 & 62 & 58 & 22 & 6672931/1164724\\
\hline
1 & 1 & 4 & 4 & 3 & 17 & 99 & 64 & 63 & 88 & 31 & 838287994463/132288736020\\
\hline
1 & 1 & 4 & 4 & 4 & 32 & 49 & 80 & 28 & 25 & 97 & 636700623/73623428\\
\hline
1 & 1 & 4 & 4 & 5 & 23 & 73 & 93 & 89 & 70 & 28 & 554259415/135115187\\
\hline
1 & 1 & 4 & 4 & 6 & 28 & 76 & 58 & 76 & 79 & 77 & 65248465/10713857\\
\hline
1 & 1 & 4 & 5 & 1 & 41 & 53 & 58 & 65 & 88 & 84 & 53835358171/11778817074\\
\hline
\end{tabular}
\newpage
\hspace*{-20mm}
\begin{tabular}{|c|c|c|c|c|c|c|c|c|c|c|c|c}
\hline
 $i=2$ & $i=3$ & $i=4$ & $i=5$ & $i=6$ &        &       &       &       &       &       &            \\   
  $j$  &  $j$  &  $j$  &  $j$  &  $j$  &  $c_2$ & $c_3$ & $c_4$ & $c_5$ & $c_6$ & $c_7$ & $\min f$   \\
 \hline
1 & 2 & 1 & 5 & 3 & 38 & 66 & 15 & 34 & 61 & 68 & 56097347111/10318159290\\
\hline
1 & 2 & 1 & 5 & 4 & 88 & 70 & 99 & 64 & 87 & 36 & 57867884661/5602942046\\
\hline
1 & 2 & 1 & 5 & 5 & 86 & 38 & 88 & 39 & 80 & 28 & 117758385/40843024\\
\hline
1 & 2 & 1 & 5 & 6 & 61 & 21 & 30 & 31 & 74 & 78 & 86072589/20358839\\
\hline
1 & 2 & 2 & 1 & 1 & 75 & 99 & 95 & 49 & 92 & 43 & 261494948913/16848566572\\
\hline
1 & 2 & 2 & 1 & 2 & 55 & 48 & 98 & 37 & 27 & 87 & 13845790001/1125978652\\
\hline
\end{tabular}
\newpage
\hspace*{-20mm}
\begin{tabular}{|c|c|c|c|c|c|c|c|c|c|c|c|c}
\hline
 $i=2$ & $i=3$ & $i=4$ & $i=5$ & $i=6$ &        &       &       &       &       &       &            \\   
  $j$  &  $j$  &  $j$  &  $j$  &  $j$  &  $c_2$ & $c_3$ & $c_4$ & $c_5$ & $c_6$ & $c_7$ & $\min f$   \\
 \hline
1 & 2 & 3 & 1 & 4 & 25 & 25 & 63 & 65 & 19 & 13 & 411775031/57007784\\
\hline
1 & 2 & 3 & 1 & 5 & 48 & 35 & 99 & 35 & 27 & 18 & 203543293/92392572\\
\hline
1 & 2 & 3 & 1 & 6 & 79 & 87 & 81 & 35 & 48 & 72 & 107294381/8210092\\
\hline
1 & 2 & 3 & 2 & 1 & 61 & 60 & 41 & 55 & 39 & 94 & 45791263395/11289763444\\
\hline
1 & 2 & 3 & 2 & 2 & 76 & 40 & 65 & 83 & 54 & 64 & 529387430832/33446610373\\
\hline
1 & 2 & 3 & 2 & 3 & 31 & 56 & 49 & 68 & 94 & 59 & 14710984249/10463343951\\
\hline
\end{tabular}
\newpage
\hspace*{-20mm}
\begin{tabular}{|c|c|c|c|c|c|c|c|c|c|c|c|c}
\hline
 $i=2$ & $i=3$ & $i=4$ & $i=5$ & $i=6$ &        &       &       &       &       &       &            \\   
  $j$  &  $j$  &  $j$  &  $j$  &  $j$  &  $c_2$ & $c_3$ & $c_4$ & $c_5$ & $c_6$ & $c_7$ & $\min f$   \\
 \hline
1 & 2 & 4 & 2 & 5 & 92 & 13 & 82 & 95 & 57 & 52 & 11499030163/15316484372\\
\hline
1 & 2 & 4 & 2 & 6 & 45 & 18 & 43 & 38 & 33 & 13 & 726699591/512577524\\
\hline
1 & 2 & 4 & 3 & 1 & 39 & 18 & 42 & 47 & 44 & 32 & 4159253753/878378948\\
\hline
1 & 2 & 4 & 3 & 2 & 97 & 90 & 66 & 52 & 90 & 98 & 287557187/72729026\\
\hline
1 & 2 & 4 & 3 & 3 & 81 & 55 & 83 & 71 & 31 & 81 & 141328284017/21352264956\\
\hline
1 & 2 & 4 & 3 & 4 & 53 & 40 & 89 & 42 & 24 & 72 & 196556065/43112142\\
\hline
\end{tabular}
\newpage
\hspace*{-20mm}
\begin{tabular}{|c|c|c|c|c|c|c|c|c|c|c|c|c}
\hline
 $i=2$ & $i=3$ & $i=4$ & $i=5$ & $i=6$ &        &       &       &       &       &       &            \\   
  $j$  &  $j$  &  $j$  &  $j$  &  $j$  &  $c_2$ & $c_3$ & $c_4$ & $c_5$ & $c_6$ & $c_7$ & $\min f$   \\
 \hline
1 & 3 & 1 & 3 & 6 & 18 & 61 & 83 & 91 & 85 & 46 & 154854949703/56313211396\\
\hline
1 & 3 & 1 & 4 & 1 & 8 & 73 & 101 & 16 & 76 & 18 & 128175437/27973638\\
\hline
1 & 3 & 1 & 4 & 2 & 66 & 78 & 80 & 22 & 43 & 83 & 2878003591/271848349\\
\hline
1 & 3 & 1 & 4 & 3 & 10 & 55 & 85 & 52 & 58 & 52 & 27041143/4706264\\
\hline
1 & 3 & 1 & 4 & 4 & 16 & 73 & 95 & 25 & 62 & 73 & 20197907861/2422621584\\
\hline
1 & 3 & 1 & 4 & 5 & 40 & 89 & 70 & 85 & 44 & 53 & 109408183921/20309105344\\
\hline
\end{tabular}
\newpage
\hspace*{-20mm}
\begin{tabular}{|c|c|c|c|c|c|c|c|c|c|c|c|c}
\hline
 $i=2$ & $i=3$ & $i=4$ & $i=5$ & $i=6$ &        &       &       &       &       &       &            \\   
  $j$  &  $j$  &  $j$  &  $j$  &  $j$  &  $c_2$ & $c_3$ & $c_4$ & $c_5$ & $c_6$ & $c_7$ & $\min f$   \\
 \hline
1 & 3 & 2 & 5 & 1 & 95 & 93 & 65 & 57 & 79 & 49 & 421679086271/677948285448\\
\hline
1 & 3 & 2 & 5 & 2 & 89 & 44 & 93 & 96 & 39 & 50 & 335620827247/34518757044\\
\hline
1 & 3 & 2 & 5 & 3 & 54 & 72 & 94 & 64 & 93 & 70 & 66025518351/8575074083\\
\hline
1 & 3 & 2 & 5 & 4 & 53 & 41 & 39 & 78 & 88 & 18 & 176408878/141602183\\
\hline
1 & 3 & 2 & 5 & 5 & 23 & 24 & 35 & 62 & 15 & 67 & 274760947/87826742\\
\hline
1 & 3 & 2 & 5 & 6 & 42 & 26 & 80 & 84 & 59 & 8 & 11381501/26544705\\
\hline
\end{tabular}
\newpage
\hspace*{-20mm}
\begin{tabular}{|c|c|c|c|c|c|c|c|c|c|c|c|c}
\hline
 $i=2$ & $i=3$ & $i=4$ & $i=5$ & $i=6$ &        &       &       &       &       &       &            \\   
  $j$  &  $j$  &  $j$  &  $j$  &  $j$  &  $c_2$ & $c_3$ & $c_4$ & $c_5$ & $c_6$ & $c_7$ & $\min f$   \\
 \hline
1 & 3 & 4 & 1 & 2 & 41 & 43 & 50 & 62 & 24 & 16 & 368838881/79661614\\
\hline
1 & 3 & 4 & 1 & 3 & 70 & 69 & 56 & 48 & 33 & 82 & 63397063/12316740\\
\hline
1 & 3 & 4 & 1 & 4 & 30 & 84 & 97 & 52 & 77 & 47 & 19693372393/41776240852\\
\hline
1 & 3 & 4 & 1 & 5 & 66 & 60 & 88 & 94 & 32 & 49 & 406657529/209912216\\
\hline
1 & 3 & 4 & 1 & 6 & 95 & 79 & 83 & 57 & 63 & 83 & 327232512461/45882008464\\
\hline
1 & 3 & 4 & 2 & 1 & 54 & 51 & 90 & 17 & 56 & 51 & 228861234/52832083\\
\hline
\end{tabular}
\newpage
\hspace*{-20mm}
\begin{tabular}{|c|c|c|c|c|c|c|c|c|c|c|c|c}
\hline
 $i=2$ & $i=3$ & $i=4$ & $i=5$ & $i=6$ &        &       &       &       &       &       &            \\   
  $j$  &  $j$  &  $j$  &  $j$  &  $j$  &  $c_2$ & $c_3$ & $c_4$ & $c_5$ & $c_6$ & $c_7$ & $\min f$   \\
 \hline
2 & 1 & 1 & 2 & 3 & 55 & 64 & 11 & 36 & 91 & 21 & 6840814043/1160344156\\
\hline
2 & 1 & 1 & 2 & 4 & 98 & 41 & 77 & 101 & 71 & 36 & 13148741828/4191371861\\
\hline
2 & 1 & 1 & 2 & 5 & 92 & 46 & 20 & 73 & 17 & 83 & 9298170127/9852554775\\
\hline
2 & 1 & 1 & 2 & 6 & 32 & 18 & 6 & 8 & 18 & 15 & 900029/505162\\
\hline
2 & 1 & 1 & 3 & 1 & 66 & 70 & 52 & 44 & 52 & 13 & 11191139/1512116\\
\hline
2 & 1 & 1 & 3 & 2 & 76 & 42 & 34 & 9 & 77 & 53 & 158864703/745294220\\
\hline
\end{tabular}
\newpage
\hspace*{-20mm}
\begin{tabular}{|c|c|c|c|c|c|c|c|c|c|c|c|c}
\hline
 $i=2$ & $i=3$ & $i=4$ & $i=5$ & $i=6$ &        &       &       &       &       &       &            \\   
  $j$  &  $j$  &  $j$  &  $j$  &  $j$  &  $c_2$ & $c_3$ & $c_4$ & $c_5$ & $c_6$ & $c_7$ & $\min f$   \\
 \hline
2 & 1 & 2 & 3 & 4 & 56 & 70 & 38 & 59 & 24 & 90 & 8161303438/919681775\\
\hline
2 & 1 & 2 & 3 & 5 & 52 & 64 & 11 & 36 & 18 & 82 & 1761926159/1038402972\\
\hline
2 & 1 & 2 & 3 & 6 & 44 & 73 & 7 & 22 & 88 & 91 & 821685949/219782284\\
\hline
2 & 1 & 2 & 4 & 1 & 90 & 40 & 53 & 28 & 42 & 18 & 21441032/3874985\\
\hline
2 & 1 & 2 & 4 & 2 & 86 & 27 & 84 & 95 & 82 & 63 & 31576124579/3988286592\\
\hline
2 & 1 & 2 & 4 & 3 & 59 & 74 & 26 & 99 & 81 & 79 & 38392434921/3718175875\\
\hline
\end{tabular}
\newpage
\hspace*{-20mm}
\begin{tabular}{|c|c|c|c|c|c|c|c|c|c|c|c|c}
\hline
 $i=2$ & $i=3$ & $i=4$ & $i=5$ & $i=6$ &        &       &       &       &       &       &            \\   
  $j$  &  $j$  &  $j$  &  $j$  &  $j$  &  $c_2$ & $c_3$ & $c_4$ & $c_5$ & $c_6$ & $c_7$ & $\min f$   \\
 \hline
2 & 1 & 3 & 4 & 5 & 52 & 68 & 28 & 95 & 100 & 15 & 8914568/2958465\\
\hline
2 & 1 & 3 & 4 & 6 & 51 & 88 & 94 & 97 & 94 & 72 & 992363501467/251755498716\\
\hline
2 & 1 & 3 & 5 & 1 & 50 & 55 & 30 & 69 & 46 & 7 & 7411013/42239804\\
\hline
2 & 1 & 3 & 5 & 2 & 59 & 84 & 14 & 59 & 40 & 61 & 31875727299/6492555482\\
\hline
2 & 1 & 3 & 5 & 3 & 21 & 84 & 4 & 44 & 72 & 78 & 3863407/1206980\\
\hline
2 & 1 & 3 & 5 & 4 & 53 & 69 & 19 & 64 & 34 & 58 & 4898280091/1950805404\\
\hline
\end{tabular}
\newpage
\hspace*{-20mm}
\begin{tabular}{|c|c|c|c|c|c|c|c|c|c|c|c|c}
\hline
 $i=2$ & $i=3$ & $i=4$ & $i=5$ & $i=6$ &        &       &       &       &       &       &            \\   
  $j$  &  $j$  &  $j$  &  $j$  &  $j$  &  $c_2$ & $c_3$ & $c_4$ & $c_5$ & $c_6$ & $c_7$ & $\min f$   \\
 \hline
2 & 1 & 4 & 5 & 6 & 22 & 68 & 26 & 88 & 56 & 61 & 1138909391/232571310\\
\hline
2 & 2 & 1 & 1 & 1 & 48 & 48 & 33 & 26 & 27 & 98 & 650136907/125880332\\
\hline
2 & 2 & 1 & 1 & 2 & 83 & 68 & 22 & 57 & 19 & 58 & 480247589/64485764\\
\hline
2 & 2 & 1 & 1 & 3 & 85 & 61 & 51 & 32 & 78 & 101 & 60903314879/34412162050\\
\hline
2 & 2 & 1 & 1 & 4 & 76 & 32 & 79 & 59 & 61 & 95 & 111402567458/27137080491\\
\hline
2 & 2 & 1 & 1 & 5 & 90 & 63 & 100 & 72 & 12 & 92 & 3042649221/345781072\\
\hline
\end{tabular}
\newpage
\hspace*{-20mm}
\begin{tabular}{|c|c|c|c|c|c|c|c|c|c|c|c|c}
\hline
 $i=2$ & $i=3$ & $i=4$ & $i=5$ & $i=6$ &        &       &       &       &       &       &            \\   
  $j$  &  $j$  &  $j$  &  $j$  &  $j$  &  $c_2$ & $c_3$ & $c_4$ & $c_5$ & $c_6$ & $c_7$ & $\min f$   \\
 \hline
2 & 2 & 2 & 2 & 1 & 94 & 77 & 90 & 30 & 49 & 36 & 29285114/2412415\\
\hline
2 & 2 & 2 & 2 & 2 & 23 & 69 & 54 & 14 & 19 & 81 & 21650521/6054361\\
\hline
2 & 2 & 2 & 2 & 3 & 17 & 92 & 15 & 25 & 43 & 52 & 185200843/78905797\\
\hline
2 & 2 & 2 & 2 & 4 & 64 & 44 & 60 & 28 & 51 & 47 & 779774470/127903991\\
\hline
2 & 2 & 2 & 2 & 5 & 66 & 12 & 93 & 96 & 42 & 34 & 349195392/119264615\\
\hline
2 & 2 & 2 & 2 & 6 & 69 & 55 & 29 & 43 & 78 & 72 & 19781146755/2684481619\\
\hline
\end{tabular}
\newpage
\hspace*{-20mm}
\begin{tabular}{|c|c|c|c|c|c|c|c|c|c|c|c|c}
\hline
 $i=2$ & $i=3$ & $i=4$ & $i=5$ & $i=6$ &        &       &       &       &       &       &            \\   
  $j$  &  $j$  &  $j$  &  $j$  &  $j$  &  $c_2$ & $c_3$ & $c_4$ & $c_5$ & $c_6$ & $c_7$ & $\min f$   \\
 \hline
2 & 2 & 3 & 3 & 2 & 30 & 69 & 15 & 72 & 75 & 22 & 166090500/43865539\\
\hline
2 & 2 & 3 & 3 & 3 & 22 & 79 & 86 & 80 & 47 & 79 & 14919718795/6934462633\\
\hline
2 & 2 & 3 & 3 & 4 & 57 & 68 & 55 & 95 & 13 & 35 & 24016133/208142772\\
\hline
2 & 2 & 3 & 3 & 5 & 38 & 76 & 70 & 85 & 56 & 24 & 412455826/104719119\\
\hline
2 & 2 & 3 & 3 & 6 & 17 & 71 & 19 & 4 & 44 & 69 & 683047157/558717398\\
\hline
2 & 2 & 3 & 4 & 1 & 49 & 32 & 53 & 73 & 41 & 12 & 59684411859/24618288760\\
\hline
\end{tabular}
\newpage
\hspace*{-20mm}
\begin{tabular}{|c|c|c|c|c|c|c|c|c|c|c|c|c}
\hline
 $i=2$ & $i=3$ & $i=4$ & $i=5$ & $i=6$ &        &       &       &       &       &       &            \\   
  $j$  &  $j$  &  $j$  &  $j$  &  $j$  &  $c_2$ & $c_3$ & $c_4$ & $c_5$ & $c_6$ & $c_7$ & $\min f$   \\
 \hline
2 & 2 & 4 & 4 & 3 & 40 & 20 & 75 & 61 & 23 & 30 & 3212989/4530364\\
\hline
2 & 2 & 4 & 4 & 4 & 23 & 43 & 65 & 59 & 61 & 74 & 27682712121/13630336306\\
\hline
2 & 2 & 4 & 4 & 5 & 43 & 44 & 94 & 73 & 65 & 100 & 27827441787/7139508580\\
\hline
2 & 2 & 4 & 4 & 6 & 15 & 33 & 51 & 96 & 70 & 28 & 3232876/4880055\\
\hline
2 & 2 & 4 & 5 & 1 & 54 & 29 & 64 & 81 & 22 & 9 & 12548376909/5835220724\\
\hline
2 & 2 & 4 & 5 & 2 & 48 & 99 & 83 & 65 & 59 & 10 & 50175516258/27264134075\\
\hline
\end{tabular}
\newpage
\hspace*{-20mm}
\begin{tabular}{|c|c|c|c|c|c|c|c|c|c|c|c|c}
\hline
 $i=2$ & $i=3$ & $i=4$ & $i=5$ & $i=6$ &        &       &       &       &       &       &            \\   
  $j$  &  $j$  &  $j$  &  $j$  &  $j$  &  $c_2$ & $c_3$ & $c_4$ & $c_5$ & $c_6$ & $c_7$ & $\min f$   \\
 \hline
2 & 3 & 1 & 5 & 4 & 40 & 96 & 55 & 17 & 35 & 55 & 124689847/27819561\\
\hline
2 & 3 & 1 & 5 & 5 & 19 & 60 & 48 & 42 & 61 & 11 & 6363/14032\\
\hline
2 & 3 & 1 & 5 & 6 & 40 & 67 & 48 & 37 & 76 & 28 & 23655234739/3369663636\\
\hline
2 & 3 & 2 & 1 & 1 & 62 & 89 & 72 & 37 & 61 & 67 & 44733143434/13199933639\\
\hline
2 & 3 & 2 & 1 & 2 & 46 & 73 & 61 & 71 & 24 & 60 & 33189245784/3571107979\\
\hline
2 & 3 & 2 & 1 & 3 & 65 & 96 & 99 & 51 & 43 & 15 & 438031831/511516422\\
\hline
\end{tabular}
\newpage
\hspace*{-20mm}
\begin{tabular}{|c|c|c|c|c|c|c|c|c|c|c|c|c}
\hline
 $i=2$ & $i=3$ & $i=4$ & $i=5$ & $i=6$ &        &       &       &       &       &       &            \\   
  $j$  &  $j$  &  $j$  &  $j$  &  $j$  &  $c_2$ & $c_3$ & $c_4$ & $c_5$ & $c_6$ & $c_7$ & $\min f$   \\
 \hline
2 & 3 & 3 & 1 & 6 & 22 & 86 & 34 & 24 & 24 & 68 & 1943219/1251552\\
\hline
2 & 3 & 3 & 2 & 1 & 25 & 81 & 37 & 41 & 15 & 26 & 486590079/217817999\\
\hline
2 & 3 & 3 & 2 & 2 & 50 & 100 & 86 & 100 & 6 & 81 & 424832800/1473008153\\
\hline
2 & 3 & 3 & 2 & 3 & 47 & 89 & 99 & 26 & 69 & 67 & 67955598311/17536306905\\
\hline
2 & 3 & 3 & 2 & 4 & 19 & 53 & 56 & 91 & 60 & 46 & 51109066899/32448120104\\
\hline
2 & 3 & 3 & 2 & 5 & 30 & 71 & 54 & 43 & 17 & 26 & 1203941826/420138499\\
\hline
\end{tabular}
\newpage
\hspace*{-20mm}
\begin{tabular}{|c|c|c|c|c|c|c|c|c|c|c|c|c}
\hline
 $i=2$ & $i=3$ & $i=4$ & $i=5$ & $i=6$ &        &       &       &       &       &       &            \\   
  $j$  &  $j$  &  $j$  &  $j$  &  $j$  &  $c_2$ & $c_3$ & $c_4$ & $c_5$ & $c_6$ & $c_7$ & $\min f$   \\
 \hline
2 & 3 & 4 & 3 & 1 & 16 & 54 & 77 & 9 & 9 & 17 & 45130047/123962708\\
\hline
2 & 3 & 4 & 3 & 2 & 61 & 85 & 74 & 29 & 92 & 51 & 170250832411/97683769736\\
\hline
2 & 3 & 4 & 3 & 3 & 10 & 79 & 33 & 56 & 43 & 8 & 12874019/11501281\\
\hline
2 & 3 & 4 & 3 & 4 & 23 & 70 & 44 & 74 & 96 & 7 & 142262209/1444323066\\
\hline
2 & 3 & 4 & 3 & 5 & 19 & 96 & 75 & 84 & 55 & 44 & 3123051271/1012763060\\
\hline
2 & 3 & 4 & 3 & 6 & 38 & 94 & 73 & 61 & 98 & 48 & 95430180809/31768686020\\
\hline
2 & 3 & 4 & 4 & 1 & 40 & 62 & 92 & 40 & 61 & 14 & 8066245783/4080225542\\
\hline
2 & 3 & 4 & 4 & 2 & 14 & 68 & 98 & 92 & 37 & 80 & 117269410/617967643\\
\hline
2 & 3 & 4 & 4 & 3 & 36 & 66 & 83 & 16 & 59 & 66 & 488279358/1187261177\\
\hline
2 & 3 & 4 & 4 & 4 & 17 & 41 & 96 & 22 & 43 & 50 & 831289141/1254382515\\
\hline
2 & 3 & 4 & 4 & 5 & 14 & 86 & 63 & 65 & 21 & 45 & 860352661/363229257\\
\hline
2 & 3 & 4 & 4 & 6 & 14 & 67 & 57 & 33 & 82 & 12 & 261074938/170478607\\
\hline
2 & 3 & 4 & 5 & 1 & 33 & 56 & 98 & 48 & 28 & 19 & 15425909/10203727\\
\hline
2 & 3 & 4 & 5 & 2 & 22 & 67 & 43 & 55 & 74 & 14 & 1421808303/1653050162\\
\hline
2 & 3 & 4 & 5 & 3 & 13 & 35 & 86 & 48 & 53 & 26 & 452098075/1946063688\\
\hline
2 & 3 & 4 & 5 & 4 & 18 & 59 & 89 & 87 & 11 & 27 & 38145640851/20849070428\\
\hline
2 & 3 & 4 & 5 & 5 & 15 & 62 & 74 & 96 & 57 & 95 & 79928085/83635768\\
\hline
2 & 3 & 4 & 5 & 6 & 28 & 66 & 70 & 94 & 29 & 19 & 1734809/61452664\\
\hline
\end{tabular}